\documentclass[10pt]{amsart}
\usepackage[utf8]{inputenc}

\usepackage{amsthm}
\usepackage{amssymb}
\usepackage[american]{babel}
\usepackage{exscale}
\usepackage[mathscr]{euscript}
\usepackage{upgreek}
\usepackage{url}
\usepackage[all,ps,tips,tpic]{xy}
\usepackage{graphicx}
\usepackage{color}
\usepackage{booktabs}
\usepackage{stmaryrd}
\usepackage{cite}
\input xy
\xyoption{all}
\usepackage[nocolor]{sseq} 
\usepackage[bookmarks=true, filecolor=green]{hyperref}
\usepackage{pdflscape}  

\newtheorem{lem}{Lemma}[section]

\newtheorem{definition}[lem]{Definition}

\newtheorem{cor}[lem]{Corollary}
\newtheorem{thm}[lem]{Theorem}
\newtheorem{prop}[lem]{Proposition}

\theoremstyle{remark}
\newtheorem{rem}[lem]{Remark}
\newtheorem{example}[lem]{Example}

\DeclareMathOperator*{\colim}{colim}

\DeclareMathOperator{\SSet}{SSet}
\DeclareMathOperator{\Hom}{Hom}
\DeclareMathOperator{\Map}{Map}

\DeclareMathOperator{\Dn}{\Delta[\textit{n}]}
\DeclareMathOperator{\Dm}{\Delta[\textit{m}]}

\DeclareMathOperator{\Sp}{Sp}
\DeclareMathOperator{\Ho}{Ho}

\DeclareMathOperator{\N}{\mathbb{N}}

\newcommand{\SHC}{\mathcal{SHC}}
\newcommand{\Top}{\mathit{Top}}
\newcommand{\C}{\mathcal{C}}
\newcommand{\D}{\mathcal{D}}
\newcommand{\U}{\mathbb{S}}

\begin{document}
 
 \title{STABLE FRAMES IN MODEL CATEGORIES}
 
 \author{Fabian Lenhardt}
 
 \maketitle
 
 \begin{abstract}
 We develop a stable analogue to the theory of cosimplicial frames in model cagegories; this is used to enrich all homotopy categories of stable model categories over the usual stable homotopy category and to give a different description of the smash product of spectra which is compared with the known descriptions; in particular, the original smash product of Boardman is identified with the newer smash products coming from a symmetric monoidal model of the stable homotopy category.
 \end{abstract}
 
\section{Introduction}

Model categories are a convenient framework for "doing homotopy theory". Despite their not very complicated definition, they are rather powerful - many of the constructions known in topology can actually be carried out in model categories or their associated homotopy categories, for example suspensions and cofiber sequences, which gives lots of extra structure one can exploit. \\
Taking a more global point of view, considering all model categories at once, one can prove the following, essentially due to Dwyer, Kan and Hovey: 

\begin{thm}
The model axioms determine the usual homotopy category of $CW$-complexes up to equivalence, including cofiber sequences and the smash product. \\
More precisely, consider for any model category $\C$ the category of left Quillen functors $\SSet \rightarrow \C$ and their natural transformations, localized at the natural weak equivalences. This is again a category which is equivalent to $\Ho(\C)$ via evaluation at the one-point simplicial set, and this is a universal property.
\end{thm}

We will explain this in some detail in Section 5. \\
This is somewhat surprising, since the model axioms are rather short and do not make any reference at all to topology - nevertheless, they completely determine the usual homotopy  category. One might say that the model category of topological spaces - or, more precisely, the Quillen equivalent model category of simplicial sets -  is  the "free model category on one generator", which yields the above theorem. We will review this in \ref{homotopyframes}.  Basically as a consequence of the precise formulation of the above theorem, one gets an enrichment of any homotopy category of a model category over the usual homotopy category. \\

In this paper, we are concerned with \textit{stable} model categories, which are those (pointed) model categories where the suspension induces an equivalence on the homotopy categories. The basic examples are the various model categories of spectra whose homotopy categories are the stable homotopy category $\SHC$; see for example \cite{BF}, \cite{EKMM}, \cite{HSS} or \cite{MMSS}. It has long been known that the category $\SHC$ is symmetric monoidal; and, as are all homotopy categories of stable model categories, it is also a triangulated category. One of the main points of this paper is that \textit{all} this structure is \textit{uniquely determined by the model axioms}:

\begin{thm}
The stable homotopy category $\SHC$ is determined up to monoidal triangulated equivalence by the model axioms.  \\
More precisely, consider for any stable model category $\C$ the category of left Quillen functors $\Sp\rightarrow \C$, where $\Sp$ is the model category of sequential spectra described in the next section, and their natural transformations, localized at the natural weak equivalences. This is again a category which is equivalent to $\Ho(\C)$ via evaluation at the sphere spectrum, and again this is a universal property.
\end{thm}

The reader may compare this to the following situation: The integers are the free group on one generator, satisfying $\Hom(\mathbb{Z}, G) \cong G$ via evaluation at $1$. Once one has this, one can reconstruct the ring structure of the integers as follows: for $a$, $b$ in $\mathbb{Z}$, choose the unique group homomorphism $f$ sending $1$ to $a$ and set $ab = f(b)$. Associativity of this is easily deduced out of the universal property. Similarly, one may construct an associative action of $\mathbb{Z}$ on any group; something very similar happens here on a categorical level to define the smash product out of the universal property. \\

In recent years, various model categories of spectra where constructed which do not only have monoidal homotopy categories, as was the case with the original definition, but which are actually symmetric monoidal on the nose. However, it is from the definitions not at all clear that the "old" and "new" definition of the smash products agree, and a proof of this seems to appear nowhere in the literature. The techniques developed here are very well-suited for giving a proof that they indeed agree (as was, of course, expected) -  the construction of the smash product given here is a new one, and it lies somewhere between the original one and the newer one, comparing nicely to both. \\
We will also obtain the following:

\begin{thm}
Let $\C$ be any stable model category. Then $\Ho(\C)$ is enriched over $\SHC$, in the sense of \cite[4.1.6]{Ho}.
\end{thm}

Similar results regarding the uniqueness of $\SHC$, disregarding the smash product, were obtained in \cite{SS}.  We will basically use the same technique, refining it somewhat. Regarding the smash product, results about monoidal uniqueness of $\SHC$ were obtained in \cite{Smonoidal}; in particular, it is proven there that all the monoidal model categories of spectra induce the same smash product on $\SHC$. We will recover this result in a stronger form by different means. Throughout the paper, our main reference will be Hovey's book on model categories \cite{Ho}, where the unstable version of the theory mentioned above is treated. However, the reader should be warned that we will adopt some conventions regarding model categories which do not agree with the ones in \cite{Ho}; see the next paragraph.

\subsection{Conventions}

Some of the following conventions may seem to be moot technical points, but since the concepts of derived functors and derived natural transformations are central to this paper, we need to be precise what we actually mean.\\

From now on, we will assume that the reader knows about model categories; good sources include \cite{DS} and, as mentioned above, \cite{Ho}. In particular, the reader should know about the model category of simplicial sets, about the homotopy category of a model category and about Quillen pairs and their derived functors; however, we will not need localisations. We will also assume fluent knowledge of basic category theory. \\
For definiteness and since there are small differences between the various sources, we adopt the definition of a model category as in \cite{DS}. Basically, this means that we do \textit{not} assume functorial factorizations or even make them part of the data as in \cite{Ho} - assuming functorial factorizations exist does not really simplify any of what follows, so there is no reason to assume them in the first place. In the end, our constructions will depend on choices; however, it boils down to choosing an inverse to an equivalence of categories, which is not really a choice; unlike in the unstable situation, this dependency on choices cannot be removed by making functorial factorizations part of the data. \\
Maybe more substantial is our convention regarding the homotopy category. We will almost exclusively be concerned with left Quillen functors (of course, there is always a right Quillen functor around, but this usually plays a background role, we just need to know that it exists); since these do not in general preserve all weak equivalences, but only those between cofibrant objects, one needs to replace cofibrantly when one wants to define derived functors. In particular, this leads to the annoying fact that the derived functor of the composition of two Quillen functors is usually not the composition of the two derived functors. We could circumvent this problem by refering to \cite[1.3.7]{Ho} which says that composition of derived functors is still coherently associative; however, we choose the following more convenient approach: For a model category $\C$, we consider the full subcategory $\Ho^{cof}(\C)$ of $\Ho(\C)$ spanned by the cofibrant objects; this category is equivalent to $\Ho(\C)$, but it has the advantage that a left Quillen functor $F: \C \rightarrow \D$ directly induces a functor $F: \Ho^{cof}(\C) \rightarrow \Ho^{cof}{\D}$ which really is $F$ on objects; by abuse of notation, we denote this functor by $F$ again. With this convention, the composition of derived functors is strict; this removes a few technically intricate issues, and we have to bother with lots of fibrant replacements already. Of course, it has the disadvantage that the definition of the derived functor of a right Quillen functor becomes more complicated; but we will never explicitly need this. 

\begin{definition}
In the remainder of this paper, we let $\Ho(\C)$ be the homotopy category of \emph{cofibrant} objects.
\end{definition}

Among the functors $\Ho(\C) \rightarrow \Ho(\D)$, there is a distinguished class: namely the derived functors of left Quillen functors. However, we will also have to consider natural transformations between derived Quillen functors, and it is less clear which ones are the "correct" class to take here. It is tempting to just take those natural transformations which are induced by natural transformations of the overlying Quillen functors; however, this has serious disadvantages; in particular it  may happen that a natural isomorphism of derived Quillen functors, stemming from a natural weak equivalence of overlying functors, is not invertible in this class. \\
We basically remove this difficulty in the crudest way possible: we allow natural transformations which are on-the-nose derived, but also the inverses to the derived natural weak equivalences and all possible compositions of these. This boils down to localizing at the natural weak equivalences; in all the cases of interest to us, we can interpret left Quillen functors and natural transformations as full subcategories of model categories with weak equivalences corresponding precisely to natural weak equivalences, which removes all set-theoretic difficulties in this.
\\
One more thing should be mentioned. In the spirit of the above discussion, a derived left Quillen functor ought to be a functor which can be written as a composition of directly derived left Quillen functors and inverses of (left) Quillen equivalences, and this is what one should usually adopt as a definition; however, in the cases relevant to us, there are sufficiently existence and uniqueness results for left Quillen functors so that it makes no difference whether one considers directly derived left Quillen functors or adopts the preceding definition, so we will stick to the former for simplicity. \\

\subsection{Acknowledgements}

This paper arose out of my diploma thesis, written under the supervision of Stefan Schwede at the University of Bonn. I would like to thank him for the supervision and for his many helpful answers to my questions, and his later help to get the material into publishable form.  I would also like to thank Arne Weiner for several fruitful discussions on the topics of this paper.

\section{Sequential spectra}
In this section, we will describe the category of spectra we want to work with. Working with this concrete model for $\SHC$ is crucial; the constructions we want to do depend not only on the homotopy category, but on the model category itself, and will fail to work for many models; in particular, we cannot replace simplicial sets by topological spaces in the following definition. 

\subsection{The category of sequential spectra}
\begin{definition}
A \emph{sequential spectrum} or just \emph{spectrum} of simplicial sets is a sequence $\{X_{n}\}_{n \geq 0}$ of pointed simplicial sets together with pointed maps $\sigma_n:  X_n \wedge S^1  \rightarrow X_{n+1}$. A map of spectra $f:X \rightarrow Y$ is a sequence of pointed maps $f_n: X_n \rightarrow Y_n$ such that the following diagram commutes for all n:
\[
\xymatrix{
				X_n \wedge S^1  \ar[rrr]^{id \wedge f_n} \ar[d]_{\sigma_n} & & & Y_n \wedge S^1 \ar[d]_{\sigma_n} \\	
				X_{n+1} \ar[rrr]_{f_{n+1}} & & & Y_{n+1}}
\]
The resulting category of sequential spectra of simplicial sets will be denoted as $\Sp$.
\end{definition}
For a simplicial set $X$, let $F_n(X)$ be the free spectrum generated by $X$ in level $n$ defined as follows: For $k < n$, $F_n(X)_k  =  *$, and for $k  > n$, $F_n(X)_k =   X \wedge S^{k-n}$. The structure maps are the  maps $X \wedge S^{k-n} \wedge S^1 \rightarrow X \wedge S^{k-n+1}$ identifying the suspension of $S^{k-n}$ with $S^{k-n+1}$. The spectrum $F_0(X)$ is called the \emph{suspension spectrum} of $X$ and is also denoted as $\Sigma^{\infty}(X)$. For a map $g: X \rightarrow Y$ of simplicial sets, there is also an induced map $F_n(g): F_n(X) \rightarrow F_n(Y)$ which is just the $(k-n)$-fold suspension of $g$ in level $k$. Thus, we have a functor $F_n: \SSet_* \rightarrow \Sp$. Also, for each $n$, we have an evaluation functor $Ev_n: \Sp \rightarrow \SSet_*$ which sends a spectrum $A$ to its $n$-th space $A_n$ and a map $f: A \rightarrow B$ of spectra to the map $f_n: A_n \rightarrow B_n$.  These functors are of course adjoint:
\begin{lem}
\label{freespectra}
$F_n: \SSet \rightleftharpoons \Sp: Ev_n$ is an adjoint functor pair.
\end{lem}
\begin{proof}
This is straightforward.
\end{proof}

The spectrum $F_0 S^0$ which has the $n$-sphere $S^n$ in level $n$ with structure maps the identifications $S^n \wedge S^1 \cong S^{n+1}$ is called the \emph{sphere spectrum} and will be denoted by $\U$. This is the spectrum which will later on play the role of the unit. \\
Sp is also a simplicial category in a natural way. For a spectrum $A$ and a simplicial set $L$, we define the spectrum $A \wedge L$ by smashing levelwise with $L$, with structure maps the structure maps of $A$ smashed with the identity of $L$.  The simplicial Hom-set of two spectra $A,B$ is the simplicial set which has in level $n$ the set $\Hom_{\Sp}(A \wedge \Dn_+, B)$, with simplicial structure maps induced by the structure maps in the cosimplicial object $A \wedge \Delta[-]$ which gives simplicial structure maps since it stands in the contravariant variable. The simplicial mapping spectrum $A^L$ is again given by applying $( -)^L$ levelwise and to the structure maps. 

\subsection{The homotopy theory of spectra}
We have a level model structure on Sp which is induced by the model structure on $\SSet_*$:  A map $f: X \rightarrow Y$ of spectra is a \emph{level weak equivalence} respectively \emph{level fibration} if all $f_n$ are weak equivalences respectively fibrations of simplicial sets, and $f$ is a cofibration if $f_0$ is a cofibration and the induced map $X_{n+1} \cup_{X_n \wedge S^1} Y_n \wedge S^1 \rightarrow Y_{n+1}$ is a cofibration for all $n$. 
 \\
For stable homotopy theory, the homotopy category of this model category is too large; speaking loosely, it should not matter what happens in low dimensions, but it certainly does for level weak equivalences. To repair this,  define the homotopy groups of a spectrum $A$ as
\[
\pi_k(A) = \colim_{n}(\pi_{k+n}|A_{n}|)
\]
for any \textit{integer} $k$ where the colimit is taken over the maps
\[
\pi_{k+n}|A_{n}| \stackrel{- \wedge S^1}{\rightarrow} \pi_{k+n+1}[A_{n}\wedge S^1| \stackrel{\sigma_{n}}{\rightarrow} \pi_{k+n+1}|A_{n+1}|
\]
A map $f: A \rightarrow B$ of spectra induces maps on the homotopy groups $\pi_{k}(f): \pi_{k}(A) \rightarrow \pi_{k}(B)$ since $f$ induces compatible maps $\pi_{k+n}|A_{n}| \rightarrow \pi_{k+n}|B_{n}|$. Call a map of spectra a \textit{$\pi_{*}$-isomorphism} if it induces isomorphisms on all homotopy groups.\\
\begin{thm}
\label{shc}
There is a model structure on $\Sp$ with weak equivalences the $\pi_*$-isomorphisms and with the same cofibrations as in the level model structure.
\end{thm}

\begin{proof}
See \cite{BF} or \cite[X]{GJ}.
\end{proof}

From now on, if we write $\SHC$, we mean $\Ho(\Sp)$. \\
The fibrant objects in this model structure are exactly the levelwise Kan $\Omega$-spectra, i.e., those spectra of levelwise Kan fibrant simplicial sets where all adjoint structure maps $A_n \rightarrow \Omega A_{n+1}$ are weak equivalences of simplicial sets, and the acyclic fibrations are the level acyclic fibrations since we did not change the cofibrations when we stabilized.  \\
Another way to obtain this stable model structure is to localize the level model structure along the maps 	$F_n(S^1) \rightarrow F_{n-1}(S^0)$ , $n \geq 0$, adjoint to the identity of $S^1$, see \cite{Hir} for details on localization; note that the level model structure on Sp has all the nice properties needed for localizing, i.e. it is cellular and left proper. The reason that the weak equivalences in this localized model structure are the $\pi_*$-isomorphisms is that every spectrum $X$ is $\pi_*$-isomorphic to an $\Omega$-spectrum, see \cite[X.4]{GJ}.  Introducing the stable model structure as this localization does not save work when trying to prove the model axioms (at least once one wants to see that weak equivalences are  exactly the $\pi_*$-isomorphisms), but the maps $F_n(S^1) \rightarrow F_{n-1}(S^0)$ will play an important role in the following; and the fact that they determine the stable model structure will be a somewhat reassuring, though not really necessary, fact. \\

Unsurprisingly, with both of the model structures defined above, Sp is a simplicial model category with the simplicial structure defined as above, i.e., the axiom SM7 holds: For a cofibration $f: A \rightarrow B$ and a fibration $g: X \rightarrow Y$, the induced map 

\[
\Map_{\Sp}(B,X) \stackrel{(f^*,g_*)}{\rightarrow} \Map_{\Sp}(A,X) \times_{\Map_{\Sp}(A,Y)} \Map_{\Sp}(B,Y)
\]
is a fibration of simplicial sets which is trivial if either $f$ or $g$ is.
Furthermore, by the levelwise definition of fibrations in the level model structure, it is evident that the adjoint pair $(F_n, Ev_n)$ is a Quillen pair for the level model structure, and since we do not change cofibrations when going to the stable model structure, it will also be a Quillen pair for the stable model structure. \\

Left Quillen functors out of the category $\Sp$ will play an important role in the following; the following lemma provides a useful criterion for deciding whether an adjoint pair $\Sp \leftrightharpoons \C$ for any model category $\C$ is Quillen.

\begin{lem}
\label{detection}
\emph{\cite[4.2]{SS}}
Let $\C$ be a model category and $F: \Sp \rightleftharpoons \C: G$ an adjoint functor pair. This adjoint pair is a Quillen functor for the stable model structure on $\Sp$ if and only if the following three conditions hold: \\[0.2cm]
i) $G$ takes acyclic fibrations to level acyclic fibrations of spectra \\
ii) $G$ takes fibrant objects to $\Omega$ - spectra\\
iii) $G$ takes fibrations between fibrant objects to level fibrations 
\end{lem}

\section{Spectra and adjunctions}
In this section, we describe why we want to work with the category Sp: It is easy to describe left adjoints starting in SSet or $\SSet_*$, and this is inherited by Sp. This is completely category-theoretical and has nothing to do with homotopy theory or model structures. Since adjoints come up again and again in the following, we adopt the convention that the direction of an adjunction is the direction of its left adjoint.  We will make heavy use of the following proposition:

\begin{prop}
Let $D$ be the cosimplicial object in $\SSet_*$ given by the pointed standard simplices $\Dn_+$ and the standard maps between them. Then for any cocomplete, pointed category $\C$, the category of adjunctions $\SSet_* \leftrightharpoons \C$ is equivalent to the category $\C^{\Delta}$ of cosimplicial objects in $\C$ via associating to an adjunction the image of $D$ under the left adjoint.
\end{prop}

\begin{proof}
The point is that $\SSet = Set^{\Delta^{op}}$ and that $\C$ is pointed. Given a cosimplicial object $T$, we define the associated right adjoint $R$ by $R(X)_n = \Hom_{\C}(T_n, X)$ with simplicial structure maps induced by the cosimplicial structure maps of $T$. The left adjoint is essentially obtained by the fact that we know what it should do on $D$ and we can glue every simplicial set in a canonical way out of the standard simplices; see \cite[3.1.6]{Ho} for details.
\end{proof}

For less awkward notation, we make the following definition:

\begin{definition}
For a cosimplicial object $X$ in $\C$, we write $(X \wedge -, \Map(X,-))$ for the associated adjunction $\SSet_* \leftrightharpoons \C$.
\end{definition}

So we know how to describe adjunctions out of $\SSet_*$ and natural transformations between them. For spectra, the point is that one can write a spectrum as a coequalizer of free spectra in a canonical way:

\begin{prop}
\label{coeq}
For a spectrum $A$, there is a coequalizer diagram
\[
\xymatrix{
			\bigvee_{n} F_n A_{n-1} \wedge S^1 \ar@<0.6ex>[rr]^-T \ar@<-0.6ex>[rr]_-H & & \bigvee_{n} F_n A_n \ar[rr] & & A}
\]
where $H$ is induced by the structure maps of $A$ and $T$ is induced by the maps $F_n A_{n-1} \wedge S^1 \rightarrow F_{n-1} A_{n-1}$ adjoint to the identity of $A_{n-1} \wedge S^1$.
\end{prop}

\begin{proof}
This is straightforward. We have maps $F_n A_n \rightarrow A$ adjoint to the identity of $A_n$, the wedge of these maps is the map $\bigvee_{n} F_n A_n \rightarrow A$. To check the universal property, note that a map $\bigvee_{n} F_n A_n \rightarrow B$ is adjoint to a sequence of maps $A_n \rightarrow B_n$; this is a map of spectra if and only if the map $\bigvee_{n} F_n A_n \rightarrow B$ is compatible with the two coequalizer maps. 
\end{proof}

Now let $\C$ be any cocomplete category and $L: Sp \rightleftharpoons \C: R$ an adjoint pair. By precomposing with the adjoint pairs $(F_n, Ev_n)$, one obtains a sequence of adjunctions $L_n: \SSet_* \rightleftharpoons \C: R_n$. Furthermore, we obtain natural transformations $\tau_n: L_{n}( - \wedge S^1) \rightarrow L_{n-1}$ as follows: For a simplicial set $K$, we have a map $F_n(K \wedge S^1) \rightarrow F_{n-1} K$ of spectra adjoint to the identity of $K$; by applying $L$, we obtain the desired natural transformations.  The adjoint natural transformation $\eta_n: R_{n-1} \rightarrow \Omega \circ R_n$ is given on an object $X$ of $\C$ as (the adjoint of) the $(n-1)$-th structure map of the spectrum $R(A)$, a map $R_{n-1} X \rightarrow \Omega R_n X$. Note that this is somewhat reminiscent of a spectrum with the $L_n$ as "spaces" and the $\tau_n$ as "structure maps".

\begin{prop}
\label{adjunction}
An adjunction $L: \C \rightleftharpoons \C: R$ is uniquely determined up to isomorphism by the $L_n$ and $\tau_n$, or equivalently the $R_n$ and $\eta_n$. Conversely, given a sequence of adjunctions $(L_n: \SSet_* \rightleftharpoons \C: R_n)$ and natural transformations $\tau_n$ as above, there is an adjunction $L: \Sp \rightleftharpoons \C: R$ inducing the $L_n$ and $\tau_n$. 
\end{prop}

\begin{proof}
We prove that $R$ is uniquely determined by the $R_n$ and $\eta_n$. For an object $X$ of $\C$, we have by definition $ev_n R(X) = R_n(X)$, and the structure map $R_{n-1} X \rightarrow \Omega R_n X$ is just $\eta_n(X)$; hence $R$ is determined up to iomorphism, which also determines $L$ up to natural isomorphism. \\
For the existence part, define a functor $R: \C \rightarrow \Sp$ as follows: For an object $X$ of $\C$, the $n$-th space of the spectrum $R(X)$ is $R_n X$, and the structure maps are $\eta_n(X): R_{n-1} X \rightarrow \Omega R_n X$. To define the left adjoint $L$, we use the coequalizer diagram in \ref{coeq}. The behaviour of a potential left adjoint on everything showing up in this coequalizer diagram is determined - it has to commute with coproducts, on free spectra $F_n K$, it has to be $L_n$, $H$ is a map induced by a map of free spectra, and $L(T)$ is a wedge on the maps $\tau_n (A_{n-1}): L_n (A_{n-1} \wedge S^1) \rightarrow L_{n-1} A_{n-1}$. Hence it is natural to define $L(A)$ via the following coequalizer diagram:

\[
\xymatrix{
			\bigvee_{n} L_n (A_{n-1} \wedge S^1) \ar@<0.6ex>[rr]^-{\bigvee L_n(\sigma_n)}\ar@<-0.6ex>[rr]_-{\bigvee \tau_n(A_{n-1})} & & \bigvee_{n} L_n A_n \ar[rr] & & L(A)}
\]
where $\sigma_n: A_{n-1} \wedge S^1 \rightarrow A_n$ denotes the structure maps of $A$.
To see that $L$ and $R$ are indeed adjoint, assume we are given a map $L(A) \rightarrow X$. This map is by definition of $L$ as the coequalizer given by a sequence of maps compatible with the coequalizer diagram $L_n A_n \rightarrow X$, which in turn are adjoint to a sequence of maps $A_n \rightarrow R_n X$; compatibilty of the maps exactly translates into this being a map of spectra $A \rightarrow R(X)$.
\end{proof}

Similiarly, we can characterize natural transformations between adjunctions $\Sp \rightleftharpoons \C$:

\begin{prop}
Let $L, L': \Sp \rightarrow \C$ be two left adjoint functors. Then a natural transformation $\phi: L \rightarrow L'$ gives rise to a sequence of natural transformations $\phi_n: L_n \rightarrow L'_n$ such that $\tau_n \circ (\phi_{n} \wedge Id_{S^1}) =  \phi_{n-1} \circ \tau_n$ as natural transformations $L_{n}( - \wedge S^1) \rightarrow L_{n-1}$. Conversely, any such sequence of natural transformations gives rise to a unique natural transformation $L \rightarrow L'$ extending the given data.
\end{prop}

\begin{proof}
This is straightforward and left to the reader.
\end{proof}

Since we know precisely how to describe left adjoints out of $\SSet_*$, we can now explicitly describe the category of adjunctions from Sp to $\C$. For a cosimplicial object, $X \wedge (- \wedge S^1): \SSet \rightarrow \C$ is again a left adjoint; we write $\Sigma X$ for the cosimplicial object associated to this adjunction; compare with section 6. Clearly, $\Sigma$ is a functor $\C^{\Delta} \rightarrow \C^{\Delta}$. 

\begin{definition}
A \emph{$\Sigma$-cospectrum} is a sequence of cosimplicial objects $X_n$ together with structure maps $\Sigma X_n \rightarrow X_{n-1}$; a morphism of cospectra $X \rightarrow Y$ is a sequence of morphisms $X_n \rightarrow Y_n$ compatible with the structure maps, like in a spectrum. Denote the resulting category as $C^{\Delta}(\Sigma)$. 
\end{definition}

The results we have obtained so far in this section directly lead to the following:  

\begin{thm}
\label{mainthm}
For a cocomplete category $\C$, the category $\C^{\Delta}(\Sigma)$ of $\Sigma$-cospectra is equivalent to the category $Ad(\Sp, \C)$ of adjunctions $\Sp \rightleftharpoons \C$ with natural transformations as morphisms.
\end{thm}

In Sp, one has a standard cospectrum, given by $F_n(D)$ in degree $n$ with $D$ the standard cosimplicial object in $\SSet_*$. The associated cospectrum to a left adjoint $\Sp \rightarrow \C$ is the image of this cospectrum in $\C$. Again, to avoid awkward notation, we make the following definition:

\begin{definition}
For a $\Sigma$-cospectrum $X$, we write $(X \wedge -, \Map(X, -))$ for the associated adjunction. We denote the $m$-th cosimplicial level of the coimplicial object $X_n$ by $X_{n,m}$.
\end{definition}

Note that none of this depended on actual properties of $\SSet_*$ or $- \wedge S^1$ besides the fact that $- \wedge S^1$ is a left adjoint; one may define $L$-cospectra in an arbitrary cocomplete category $\C$ with an adjunction $L: \C \rightleftharpoons \C: R$ as sequences of objects $X_n$ of $\C$ with structure maps $L X_n \rightarrow X_{n+1}$, and all of the above remains true for this category of spectra. In all cases of interest of us - in particular for $\Sigma$-cospectra - the underlying category $\C$ is actually a model category and $L$ is left Quillen, and then cospectra form a model category again. Since this will be important to us, we formalize the definition:

\begin{definition}
Let $\C$ be a model category and let $L: \C \rightleftharpoons \C: R$ be a Quillen pair. A \emph{cospectrum} with respect to this data is a sequence $X_0, X_1, \dots$ of objects of $\C$ together with structure maps $\sigma_n: LX_n \rightarrow X_{n-1}$. A morphism $f:X \rightarrow Y$ of cospectra is a sequence of maps $f_n: X_n \rightarrow Y_n$ such that for all n, the diagram
\[
\xymatrix{
		LX_n \ar[rrr]^{\sigma_n} \ar[d]_{Lf_n} & & & X_{n-1} \ar[d]^{f_{n-1}}\\
		LY_n \ar[rrr]_{\sigma_n} & & & Y_{n-1}
}
\]
commutes. We denote this category by $\C(L)$ and will call the objects $L$-cospectra.
\end{definition}
We will also consider the category of cospectra up to degree $k$:

\begin{definition}
With notation as above, an \emph{L-cospectrum up to degree $k$} is a sequence of objects $X_0$, $X_1$,\dots,$X_k$ in $\C$ together with structure maps $L X_m \rightarrow X_{m-1}$ for $m = 1 \ldots k$. A morphism $X \rightarrow Y$ is a sequence of maps $X_m \rightarrow Y_m$ compatible with the structure maps as above. We denote the resulting category as $\C(L, k)$.
\end{definition}
Both these constructions again yield model categories in a natural way:

\begin{thm}
\label{cospectra}
There is a level model structure on $\C(L)$ and on $\C(L,k)$ for any $k$ where a map $f: X \rightarrow Y$ is a 
\begin{itemize}
\item weak equivalence resp. cofibration if and only if all $f_n$ are weak equivalences resp. cofibrations 
\item a (trivial) fibration if $f_0$ is and for all $n \geq 0$, the induced map $X_n \rightarrow Y_n \times_{RY_{n-1}} RX_{n-1}$ is a (trivial) fibration. 
\end{itemize}
\end{thm}

\begin{proof}
This is certainly not new; however, there seems to be no actual proof of this in printing. Note that we do not assume that $\C$ is cofibrantly generated; the theorem holds for any model category. However, the proof of the model axioms, just using the model axioms in $\C$, is straightforward (though not precisely short); note that one needs to check that "trivial fibration" as stated is indeed the same as a fibration and a weak equivalence.
\end{proof}

\section{Stable frames in simplicial model categories}

In this section, we prove a slight variant of our main theorem for simplicial model categories; this serves as a good warm-up to the more complicated general case and is sufficient for many purposes. Readers interested in topological model categories should note that, replacing Sp by topological spectra (in some suitable closed monoidal category of topological spaces like k-spaces) and "simplicial" by "topological" in all that follows in this section, the proofs go through virtually unchanged. We refer to \cite{GJ} for a definition and basic facts concerning simplicial model categories.

In this section we will only be concerned with functors compatible with the simplicial structure in the following sense:

\begin{definition}
Let $\C$, $\D$ be pointed simplicial (model) categories. An adjunction $L: \C \rightleftharpoons \D: R$ together wit natural isomorphisms $L(X \wedge K) \cong L(X) \wedge K$ for any object $X$ of $\C$ and any simplicial set $K$; or equivalently natural isomorphisms $R(Y^K) \cong R(Y)^K$ for $Y$ in $\D$; is called \emph{strong simplicial} if the natural isomorphisms make two coherence diagram commutative: One equating the two ways of getting from $L(X) \wedge (K \wedge L)$ to $L((X \wedge K) \wedge L)$ and one equating the two ways of getting from $L(X) \wedge \Delta[0]$ to $L(X)$; or equivalently, if the corresponding coherence diagrams for $R$ commute.
\end{definition}

Given a cocomplete pointed simplicial category $\C$, we now want to describe the strong simplicial adjunctions $L: \Sp \rightleftharpoons \C: R$. If $(L,R)$ is strong simplicial, then it follows that $(L_n = L \circ F_n, R_n = Ev_n \circ R)$ is a strong simplicial adjunction $\SSet_* \rightleftharpoons \C$ - and such an adjunction is completely determined up to isomorphism by the image of $S^0$, an object of $\C$. Similarly, the natural transformation $\tau_n: L_n \circ(- \wedge S^1) \rightarrow L_{n-1}$ is a simplicial natural transformation, hence simply given by a morphism in $\C$ and we obtain a $S^1$-cospectrum in $\C$. Similar arguments can be made for natural transformations, leading to the following:

\begin{prop}
Let $\C$ be a pointed simplicial model category. Then the category of strong simplicial adjunctions $\Sp \rightleftharpoons \C$ and simplicial natural transformations is equivalent to the category $\C(S^1)$ of $S^1$-cospectra. 
\end{prop}

In this section, we will drop the $S^1$, just speaking of cospectra. \\
Given such a cospectrum $X$, we can describe the associated right adjoint as follows: The $n$-th space of $R(Y)$ is $\Map(X_n, Y)$, where $\Map$ denotes the simplicial mapping space in $\C$, with structure maps adjoint to $\Map(X_{n-1}, Y) \rightarrow \Map(X_n \wedge S^1,Y) \cong \Omega \Map(X_n, Y)$.  We will denote the adjoint pair associated to $X$ by $(X \wedge -, \Map(X,-))$. Given a simplicial left adjoint $L: \Sp \rightarrow \C$, we obtain the associated cospectrum $X$ by $X_n = L(F_n S^0)$ with structure maps the images of the canonical maps $F_n S^1 \rightarrow F_{n-1} S^0$. \\
Now we want to apply \ref{detection} to find necessary and sufficient conditions for a cospectrum $X$ to present a Quillen pair. For part i) of \ref{detection}, we note that for an acyclic fibration $f$ in $\C$, $\Map(X, f)$ is levelwise of the form $\Map(X_n, f)$; the only natural condition ensuring that $\Map(X, f)$ is level acyclic is that all $X_n$ are cofibrant, which is also necessary since $X_n = X \wedge F_n S^0$ needs to be cofibrant if $X \wedge -$ is left Quillen. For ii), note that for fibrant $Y$, $\Map(-, Y)$ preserves weak equivalences between cofibrant objects, hence the structure map $\Map(X_{n-1}, Y) \rightarrow \Map(X_n \wedge S^1,Y) \cong \Omega \Map(X_n, Y)$ would be a weak equivalence if $X_{n} \wedge S^1 \rightarrow X_{n-1}$  were a weak equivalence. This is also necessary since this is the image of the canonical map $F_n S^1 \rightarrow F_{n-1} S^0$ which is a $\pi_*$-isomorphism, being an isomorphism in all but one degrees, and this needs to be preserved if $X \wedge -$ is left Quillen. Also note that $\Map(X,Y)$ is automatically levelwise Kan if $Y$ is fibrant and $X$ cofibrant. Part iii) is no new condition since fibrations induce fibrations $\Map(X_n, f)$ anyway. 
Hence we arrive at the following definition:

\begin{definition}
A cospectrum $X$ is a \emph{simplicial stable frame} if it is pointwise cofibrant with weak equivalences $X_n \wedge S^1 \rightarrow X_{n-1}$ as structure maps
\end{definition} 

The discussion above yields:

\begin{prop}
\label{stableframing}
The cospectrum $X$ represents a Quillen pair $\Sp \leftrightharpoons \C$ if and only if it is a simplicial stable frame.
\end{prop}

Now we have a good combinatorial handle on left Quillen functors out of $\Sp$ into a stable simplicial model category $\C$, we may ask whether such functors exist. It turns out there are plenty of them:

\begin{prop}
\label{existencestable}
Let $\C$ be a stable simplicial model category and $X$ a cofibrant-fibrant object of $\C$. Then there is a simplicial stable frame $A$ with $A_0 = X$. Additionally, one can arrange $A$ to be fibrant in the model structure on cospectra, i.e., such that $A$ is level fibrant and the structure maps $A_n \rightarrow \Omega A_{n-1}$ are acyclic fibrations. Thus there is a left Quillen functor $A \wedge -: \Sp \rightarrow \C$ with $A \wedge \U \cong X$.
\end{prop}

\begin{proof}
Since $X$ is fibrant, also $\Omega X$ is fibrant. Choose a cofibrant-fibrant object $A_1$ together with an acyclic fibration $A_1 \rightarrow \Omega X$  by factoring $* \rightarrow \Omega X$ into a cofibration followed by an acyclic fibration. Since $\C$ is stable, $(S^1, \Omega)$ is a Quillen equivalence. Thus, since $A_1$ and $X$ are cofibrant-fibrant, the adjoint map $A_1 \wedge S^1 \rightarrow X$ is also a weak equivalence. Repeating this, one gets cofibrant-fibrant objects $A_n$ with weak equivalences $A_n \wedge S^1 \rightarrow A_{n-1}$ whose adjoints are fibrations. Note that, although the factorizations provide fibrant frames, not all frames need to be fibrant, as in the following example. However, for fibrant stable frames, also the adjoint structure maps $X_n \rightarrow \Omega X_{n-1}$ are weak equivalences since $(S^1, \Omega)$ is a Quillen equivalence; this is the main reason that the simplicial case is easier to handle.
\end{proof}

\begin{example}
\label{SymSpec}
 Let $Sp^{\Sigma}$ be the category of symmetric spectra of simplicial sets as introduced in ~\cite{HSS} with the injective stable model structure \cite[3.4]{HSS}. We have the forgetful functor $R: \Sp^{\Sigma} \rightarrow \Sp$ which just forgets the symmetric group actions. This functor commutes with simplicial mapping spaces and with limits, so it has a chance to be the right adjoint of a strong simplicial adjunction. Let us see where a possible left adjoint $L$ should map the $F_n(S^0)$.
Our various adjunctions give
\[
\Sp^{\Sigma}(L(F_n(S^0)), X) \cong \Sp(F_n(S^0), RX) \cong \SSet_*(S^0, (RX)_n) \cong \Sp^{\Sigma}(F_n^{\Sigma}(S^0), X)
\]
 where $F_n^{\Sigma}(S^0)$ is the free symmetric spectrum on $S^0$ in dimension $n$, see \cite[2.5]{HSS}. The Yoneda lemma forces that $L(F_n(S^0)) \cong F_n^{\Sigma}(S^0)$, so let us see how we can extend the $F_n^{\Sigma}(S^0)$ to a stable frame. The obvious candidate for the structure maps are the maps $F_n^{\Sigma}(S^0) \wedge S^1 \rightarrow F_{n-1}(S^0)$ adjoint to the identity of $S^1$. These are stable equivalences in the stable model structure on $\Sp^{\Sigma}$, see \cite[3.1.10]{HSS}, and all free symmetric spectra are cofibrant.  So we have indeed a Quillen pair $(F,G)$ induced by this stable framing; one easily checks that $G$ is indeed the forgetful functor.
\end{example}

\subsection{Uniqueness of stable frames}
Of course, in general there are many choices to build a stable frame on a cofibrant-fibrant object in a stable model category, since already in the level fibrant framings constructed in Theorem \ref{existencestable} we had many choices of factoring the maps, and  there may be  more non-level fibrant stable frames. But "up to homotopy", they are unique, as we will see in this section.

\begin{prop}
\label{fullfaithful}
Let $f: X_0 \rightarrow Y_0$ be a morphism in a simplicial stable model category $\C$ with $X_0$ cofibrant and $Y_0$ cofibrant-fibrant, and let $X$ be a simplicial stable frame on $X_0$, $Y$ a fibrant simplicial stable frame on $Y_0$. Then there is a map $F: X \rightarrow Y$ extending $f$, and any two such extensions are homotopic in $\C(S^1)$. Additionally, if $f$ was a weak equivalence, so is $F$.
\end{prop}

\begin{proof}
The forgetful functor from cospectra to $\C$ sending a cospectrum $Z$ to $Z_0$ has a right adjoint $\Omega^{\infty}$, given on an object $A$ of $\C$ by $\Omega^{\infty}(A)_n = \Omega^n A$ and structure maps $\Omega^n A \wedge S^1 \rightarrow \Omega^{n-1} A$ adjoint to the identity of $ \Omega^n A$.  The counit $Z \rightarrow \Omega^{\infty} Z_0$ is easily seen to be a fibration of cospectra if the cospectrum $Z$ is fibrant, and a weak equivalence if $Z$ is additionally a stable frame.  Now we can choose a lift in the diagram

\[
\xymatrix{
		& & Y \ar[dd] \\
		\\
		X \ar[rr]^{\tilde{f}} & & \Omega^{\infty} Y_0
}
\]

This is the desired $F$. It is also clear that there is only one lift up to homotopy since two different lifts become equal after composition with the weak equivalence $Y \rightarrow \Omega^{\infty} Y_0$, hence are equal in the homotopy category after composition with an isomorphism. If $f$ is a weak equivalence, it is easy to check that the suspension of $F_1$ is a weak equivalence; since $F_1$ has cofibrant source and fibrant target and the model category is stable, this means $F_1$ is a weak equivalence. Arguing inductively, it follows that $F$ is a weak equivalence.
\end{proof}

This seems to be very satisfying; however, one needs to be careful here: Although all maps covering $f$ as above are homotopic, it is not clear that \emph{homotopic maps induce the same derived natural transformation}. To prove this, the model structure on cospectra introduced in \ref{cospectra} will be used.

\begin{prop}
\label{compatiblesimp}
Let $\C$ be a model category. Assume $f: X \rightarrow Y$ is a cofibration in $\C(S^1)$  and $g: A \rightarrow B$ is a cofibration of spectra. Then the induced pushout-product map in $\C$ $f  \Box g: X \wedge B \coprod_{X \wedge A} Y \wedge A \rightarrow Y \wedge B$ is a cofibration which is trivial if either $f$ or $g$ is.
\end{prop} 

\begin{proof}
We may assume that $g$ is one of the generating cofibrations $F_m \partial \Dn \rightarrow F_m \Dn$ by \cite{Ho}[4.2.4]. In this case, we have $X \wedge F_m \Dn = X_m \wedge \Dn$ and $X \wedge F_m \partial \Dn = X_m \wedge \partial \Dn$, and similarly for $Y$.  Now, the conclusion follows from the SM7-axiom for $\C$. 
\end{proof}

 \begin{cor}
 \label{QQS}
 Let $B$ be a cofibrant spectrum. Then the functor $- \wedge B: \C(S^1) \rightarrow \C$ preserves cofibrations and acyclic cofibrations and hence has a left derived functor. 
 \end{cor}
 
 \begin{proof}
We set $A = *$ in \ref{compatiblesimp}. The pushout-product map is then just the map $f \wedge B: X \wedge B \rightarrow Y \wedge B$ which is an (acyclic) cofibration if $f$ is. Hence $- \wedge B$ preserves weak equivalences between cofibrant objects and has a left derived functor as claimed.
\end{proof}

\begin{cor}
Let $X$, $Y$ be simplicial stable frames and $F, G: X \rightarrow Y$ two homotopic maps. Then $F$ and $G$ induce the same derived natural transformations between the derived functors of $X \wedge -$ and $Y \wedge -$. 
\end{cor}

\begin{proof}
Let $A$ be a cofibrant spectrum. We have two maps $F(A), G(A): X \wedge A \rightarrow Y \wedge A$. We claim that these two maps represent the same map in $\Ho(\C)$: Since $- \wedge A$ has a derived functor by the preceding corollary, we get a diagram of functors
\[
\xymatrix{
		\C(S^1) \ar[d] \ar[rr]^{- \wedge A} & & \C \ar[d] \\
		\Ho(\C(S^1) \ar[rr]^{- \wedge^L A} & & Ho(\C)
}
\]
which commutes up to a natural isomorphism.  We want to see that $F$ and $G$ go to the same map via the clockwise composition; but since the left vertical map sending them to their homotopy classes already  sends them to the same map, the clockwise composition, being isomorphic to the counterclockwise one, then has to send these two maps to the same map as well.  So for all cofibrant spectra $A$, $F(A), G(A): X \wedge A \rightarrow Y \wedge A$ represent the same map in $\Ho(\C)$, and this is enough to conclude that the derived natural transformations of $F$ and $G$ are equal.
\end{proof}

The next statement tells us that natural weak equivalences of Quillen functors correspond precisely to weak equivalences of simplicial stable frames:

\begin{prop}
Let $f: X \rightarrow Y$ be a weak equivalence of stable frames. Then the derived natural transformation $f \wedge -: X \wedge - \rightarrow Y \wedge -$ is a natural weak equivalence, i.e., a weak equivalence for all cofibrant spectra $A$.  Conversely, if $f \wedge -$ is a natural weak equivalence, then $f$ is a weak equivalence.
\end{prop}

\begin{proof}
By \cite[1.3.18]{Ho}, we may as well check the corresponding statement for the right adjoints, i.e., that for a fibrant object $Z$ of $\C$, the map 
\[
\Map(f, Z): \Map(Y, Z) \rightarrow \Map(X,Z)
\]
is a $\pi_*$-isomorphism. Since the simplicial mapping space $\Map(-, Z)$ is left Quillen considered as a functor $\C \rightarrow \SSet^{op}$, it preserves weak equivalence between cofibrant objects; in particular, the weak equivalence $f_n$ induces a weak equivalence $\Map(Y_n, Z) \rightarrow \Map(X_n, Z)$ and hence a level weak equivalence $\Map(f, Z): \Map(Y, Z) \rightarrow \Map(X,Z)$ as claimed.  Note that since the involved spectra are all fibrant, level equivalences and $\pi_*$-isomorphisms are equal. For the converse, note that $f_n = f \wedge F_n \Delta[0]$ is a weak equivalence if $f \wedge -$ i a natural weak equivalence.
\end{proof}

This allows us to describe the category of derived simplicial Quillen functors from $\Sp$ to $\C$:

\begin{definition}
For a stable model category $\C$, let $SSF(\C)$ denote the full subcategory of $\C(S^1)$ given by all stable frames and $\Ho(SSF(\C))$ the full subcategory of the homotopy category of $\C(S^1)$ given by stable frames.
\end{definition}

\begin{prop}
The category $\Ho(SSF(\C))$ is equivalent to the category of derived simplicial Quillen functors $\Sp \rightarrow \C$.
\end{prop}

\begin{proof}
The category $SSF(\C)$ is equivalent to the category of simplicial Quillen functors $\Sp \rightarrow \C$. To obtain the category of derived functors, we localise at the natural weak equivalences; under the equivalence of categories, the preceding proposition tells us that this corresponds to inverting the weak equivalences of cospectra, which yields the homotopy category $\Ho(SSF(\C))$. In particular, we have an actual \emph{category} of derived Quillen functors.
\end{proof}

So we can interpret $\Ho(SSF(\C))$ as  a category of functors $\SHC \rightarrow \Ho(\C)$, and we will usually do so. The central theorem of this section tells us that this category is nothing new:

\begin{thm}
Evaluation in degree 0 induces an equivalence of categories 
\[
ev_0: \Ho(SSF(\C)) \rightarrow \Ho(\C)
\]
\end{thm}

\begin{proof}
Since evaluation in degree $0$ is left Quillen, we indeed obtain a derived functor on the homotopy  categories. That $ev_0$ is full is the statement of \ref{existencestable}: For each object $X$ of $\Ho(\C)$, there is an isomorphic object $Y$ - for example, a fibrant replacement of $X$ - on which we can build a stable frame. The fact that $ev_0$ is full and faithful is the content of \ref{fullfaithful}: We can extend any map in degree $0$ to a map of stable frames if we have chosen the frame appropriately, which is always possible, and such an extenion is unique up to homotopy.
\end{proof}

This theorem can be fed into the constructions in Chapter 7 instead of Theorem \ref{MAIN}, and one easily obtains simplicial analogues of the theorems there. Also note that allowing arbitrary Quillen functors $\Sp \rightarrow \C$ does not really change anything: The category of derived functors is still equivalent to $\Ho(\C)$; hence, in the simplicial case, there is no real loss of generality by only considering strong simplicial functors.

\section{Cosimplicial frames}

The proof of the general case of our main theorem relies on the technique of frames in a model category first developed in \cite{DK}; we will give a short overview, mostly based on \cite[Section 5]{Ho}.  \\
There is one easy definition of a cosimplicial frame: A cosimplicial object in a (pointed) model category $\C$ is a frame if and only if the associated adjunction $\SSet_* \rightleftharpoons \C$ is a Quillen pair. This is the correct definition; however, we will give an equivalent description of frames which is easier to handle since it makes no reference to the associated adjunction; only intrinsic properties of the cosimplicial object will be used. 
\subsection{Additional structure on $\C^{\Delta}$}
Let $\C$ be a pointed, cocomplete category. The category of cosimplicial objects in $\C$ is a simplicial category in a natural way: 

\begin{definition}
For a cosimplicial object $X$ in $\C$ and a pointed simplicial set $K$, define a cosimplicial object $X \wedge_S K$ by
\[
(X \wedge_S K)_n = X \wedge(K \wedge \Dn_+)
\]
with cosimplicial structure maps induced by the cosimplicial structure map of the cosimplicial object $K \wedge \Delta [-]_+$ under the functor $X \wedge -$. 
\end{definition}

Using the equivalent language of adjunctions $\SSet_* \rightleftharpoons \C$, this construction takes the following form: A cosimplicial object $X$ represents an adjunction $\SSet_* \rightleftharpoons \C$, and we can precompose this adjunction with the adjunction $(K \wedge -, (-)^K): \SSet_*\rightleftharpoons \SSet_*$ to obtain another adjunction $\SSet_* \rightleftharpoons \C$, and this adjunction is represented by the cosimplicial object $X \wedge_S K$. 

\begin{prop}
This smash product is part of a simplicial structure on $\C^{\Delta}$.
\end{prop}

\begin{proof}
The simplicial mapping spaces are defined as
\[
\Map(X,Y)_n = \Hom_{\C^{\Delta}} (X \wedge_S \Dn_+, Y)
\]
with simplicial structure maps induced from the cosimplicial structure maps of the cosimplicial object of cosimplicial objects $X \wedge_S \Delta [-]_+$. See \cite[II.1]{Qui} for a description of the right adjoints $(-)^K: \C^{\Delta} \rightarrow \C^{\Delta}$.
\end{proof}

If $\C$ is a model category, the category $\C^{\Delta}$ also carries a model structure. Note the different meanings of $X \wedge_S K$ and $X \wedge K$: The first is a cosimplicial object, the latter an object of $\C$. Also recall that $X \wedge \Dn_+  \cong X_n$.

\begin{definition}
Let $f:X \rightarrow Y$ be a morphism of cosimplicial objects in a model category $\C$. The map $f$ is
\begin{itemize}
\item a weak equivalence if for all $n$, the map $f_n: X_n \rightarrow Y_n$ is a weak equivalence
\item a (acyclic) Reedy cofibration if the induced maps 
\[
X \wedge \Dn_+ \coprod_{X \wedge \partial \Dn_+} Y \wedge \partial \Dn_+ \rightarrow Y_n
\]
 are (acyclic) cofibrations in $\C$ for all $n$
\item a (acyclic) Reedy fibration if it has the corresponding right lifting property with respect to (acyclic) Reedy cofibrations 
\end{itemize}
\end{definition}

\begin{rem}
It is easy to check with this definition that an object $Y$ is Reedy fibrant if and only if the induced map $Y \rightarrow c(Y_0)$  is a Reedy fibration and $Y_0$ is fibrant; we will often make use of this.
\end{rem} 

Of course, this defines a model structure such that the two possibly different notions of acyclic cofibrations (resp. the two possibly different notions of acyclic fibrations) agree:
\begin{thm}
\label{reedy}
With these classes of weak equivalences, cofibrations and fibrations, the category $\C^{\Delta}$ is a model category.
\end{thm}
\begin{proof}
This is \cite[5.2.5]{Ho}, or see \cite{Reedy}.
\end{proof}

Note that a cosimplicial object $X$ is cofibrant if and only if $X \wedge -$ preserves cofibrations - by definition, it preserves the generating cofibrations $\partial \Dn_+ \rightarrow \Dn_+$, and hence all cofibrations. \\
Now $C^{\Delta}$ is a model category and a simplicial category; however, the SM7-axiom for a simplicial model category fails. We only have the following:

\begin{prop}
\label{SM7}
Let $f: X \rightarrow Y$ be a Reedy cofibration of cosimplicial objects in a model category $\C$ and let $i: K \rightarrow L$ be a cofibration of simplicial sets. Then the pushout-product map 
\[
f \Box i: X \wedge_S L \coprod_{X \wedge_S K} Y \wedge_S K \rightarrow Y \wedge_S L
\]
is a cofibration which is trivial if $f$ is.
\end{prop}

\begin{proof}
See \cite[7.4]{RSS} or \cite[5.4.1]{Ho} (or rather its pointed analogue \cite[5.7.1]{Ho}); or use the methods in the proof of Proposition \ref{compatible}.
\end{proof}

This is close to the SM7-axiom, but the pushout-product need not be a weak equivalence when $i$ is a trivial cofibration.  \\
Now we can characterize frames:

\begin{prop}
\label{frames}
Let $X$ be a cosimplicial object in a model category $\C$. Then $X \wedge -: \SSet \rightarrow \C$ is left Quillen if and only if $X$ is Reedy cofibrant and for all standard maps $\Dn_+ \rightarrow \Dm_+$ of standard simplices, $X \wedge \Dn_+ \rightarrow X \wedge \Dm_+$ is a weak equivalence in $\C$.
\end{prop}

\begin{proof}
We have seen in the above discussion that $X \wedge -$ preserves cofibrations if and only if $X$ is Reedy cofibrant. If $X \wedge -$ is left Quillen, it preserves weak equivalences between cofibrant objects by Ken Brown's Lemma \cite[1.1.12]{Ho} and hence $X \wedge \Dn_+ \rightarrow X \wedge \Dm_+$ is a weak equivalence in $\C$. The converse is proven in \cite[3.6.8]{Ho}: If $X \wedge -$ preserves cofibrations and the maps $X \wedge \Dn_+ \rightarrow X \wedge \Dm_+$ are weak equivalences, $X \wedge -$ preserves acyclic cofibrations. 
\end{proof}

Note that $X \wedge \Dn_+ \rightarrow X \wedge \Dm_+$ if and only if all cosimplicial structure maps of $X$ are weak equivalences if and only if the induced map $X \rightarrow c(X_0)$ is a level weak equivalence. This leads to the following definition:

\begin{definition}
A cosimplicial object in $\C$ is \emph{homotopically constant} if all cosimplicial structure maps are weak equivalences. It is a \emph{cosimplicial frame} or just \emph{frame} if it is Reedy cofibrant and homotopically constant. 
\end{definition}

By the above proposition, frames correspond to Quillen pairs $\SSet_* \rightleftharpoons \C$. We also have the following:

\begin{prop}
For any simplicial set $K$ and any frame $X$, the cosimplicial object $X \wedge_S K$ is again a frame.
\end{prop}

\begin{proof}
A cosimplicial object $X$ is a frame if and only if the corresponding adjunction is Quillen. Since the adjunction $(K \wedge -, (-)^K): \SSet_* \rightleftharpoons \SSet_*$ is Quillen, the composite of this adjunction with $(X \wedge -, \Map(X,-))$ is Quillen if $X$ is a frame; hence $X \wedge_S K$ is a frame.
\end{proof}
Since we will mainly use smashing with $S^1$, we introduce simpler notation:

\begin{definition}
\label{Sigma}
For a cosimplicial object $X$, we write $\Sigma X$ for the cosimplicial object $X \wedge_S S^1$ and $\Omega (-)$ for the right adjoint of $\Sigma$.
\end{definition}

For frames, the simplicial mapping spaces in $\C^{\Delta}$ also carry homotopical information (in general, they do not because SM7 fails):

\begin{prop}
\label{mapping}
For cosimplicial objects $A$ and $B$ in $\C$ with $A$ a cosimplicial frame and $B$ Reedy fibrant, we have a natural isomorphism
\[
\pi_n \Map(A,B) \cong [A \wedge_S S^n, B]
\]
where $[-,-]$ denotes the morphism sets in $\Ho(\C^{\Delta})$.
\end{prop}

\begin{proof}
The point is that $A \wedge_S \Delta[-]$ is a cosimplicial frame on $A$  (it is a bicosimplicial object in $\C$) if $A$ is a frame. Using Proposition \ref{SM7}, this is easy to see: The map $A \wedge_S \partial \Dn \rightarrow A \wedge \Dn$ is a cofibration by setting $X = *$, $Y = A$, $K = \partial \Dn$ and $L = \Dn$ in \ref{SM7}. The maps $A \wedge_S \Dn_+ \rightarrow A \wedge_S \Delta[m]_+$ coming from maps $\Dn_+ \rightarrow \Delta[m]_+$ are levelwise of the form $A \wedge( \Dn_+ \wedge \Delta[k]_+) \rightarrow A \wedge( \Delta[m]_+ \wedge \Delta[k]_+)$ and hence are weak equivalences since $A \wedge -$ is left Quillen and preserves weak equivalences between cofibrant objects. Hence the weak equivalences $\Dn_+ \rightarrow \Delta[m]_+$ are preserved, so $A \wedge_S \Delta[-]_+$ is homotopically constant. The claim now follows from \cite[6.1.2]{Ho} which states that the mapping spaces obtained from frames have the correct homotopy type.
\end{proof}

\subsection{Frames}

Now we have defined frames and seen some basic properties, we want to put together some results concerning existence and uniqueness of frames. Basically, we want to prove the following uniqueness theorem:

\begin{thm}
\label{homotopyframes}
Let $\Ho(Fr(\C))$ be the full subcategory of $\Ho(\C^{\Delta})$ with objects the cosimplicial frames. Then evaluation in degree $0$ $ev_0: \C^{\Delta} \rightarrow \C$ induces an equivalence of categories $\Ho(Fr(\C)) \rightarrow \Ho(\C)$. Furthermore, the suspension functor $\Sigma: \C^{\Delta} \rightarrow \C^{\Delta}$ restricts to a functor $\Sigma: \Ho(Fr(\C)) \rightarrow \Ho(Fr(\C))$ which is an equivalence if $\C$ is stable.
 \end{thm}

For this end, we need to develop some more theory. This theorem subsumes most of the properties of frames in a very compact form; it is not formulated in \cite{Ho}, but all the ingredients for the proof can be found there. \\

The first question we want to answer is that of existence of frames: Given a cofibrant object $X$ of $\C$, can we find a frame $A$ with $A_0 =  X$? This is answered by the following proposition:

\begin{prop}
\label{existenceframings}
Let $\C$ be a pointed model category and $X$ a cofibrant object of $\C$. Then there is a left Quillen functor $L: \SSet_* \rightarrow \C$ with an isomorphism from $L(\Delta[0])$ to $X$; or equivalently, there is a frame $A$ on $X$, i.e., such that $A_0 \cong X$. If $X$ is fibrant, we may choose $A$  to be Reedy fibrant.
\end{prop}

\begin{proof}
This is basically a consequence of the factorizations in $\C^{\Delta}$; see \cite[5.2.8]{Ho}.
\end{proof}

\begin{prop}
\label{framescover}
Let $X, Y$ be cofibrant-fibrant  objects of $\C$, $f: X \rightarrow Y$ a morphism in $\C$. Then there are frames $A$ on $X$, $B$ on $Y$ together with a morphism $F: X \rightarrow Y$ covering $f$.
\end{prop}

\begin{proof}
This is essentially \cite[5.5.1]{Ho}.
\end{proof}

\begin{prop}
\label{uniqueframe}
Let $A, B$ be frames and $f, g: A \rightarrow B$ be two maps such that $f_0, g_0: A_0 \rightarrow B_0$ represent the same morphism in $\Ho(\C)$. Then $f = g$ in $\Ho(\C^{\Delta})$.
\end{prop}

\begin{proof}
This is similar to \cite[5.5.2]{Ho}, though not quite the same. Let $ev_0: \C^{\Delta} \rightleftharpoons \C: c$ denote the adjunction given by evaluation in degree $0$ and constant cosimplicial object. This is a Quillen pair by the discussion after \cite[5.2.7]{Ho}. The assumptions of the proposition are such that for the derived functor $ev_0^L$, we have $ev_0^L(f) = ev_0^L(g)$. Since $B$ is homotopically constant, the map $B \rightarrow c(ev_0(B))$ is a weak equivalence; hence the counit of the derived adjunction is an isomorphism for frames. Since $c^R(ev_0^L(f)) = c^R(ev_0^L(g))$, it follows $f = g$ in $\Ho(\C^{\Delta})$ as desired.
\end{proof}

We will also need the following:

\begin{prop}
Let $A, B$ be frames, $f,g: A \rightarrow B$ maps which are equal in $\Ho(\C^{\Delta})$. Then the derived natural transformations $f^L, g^L: A \wedge^L - \rightarrow B \wedge^L -$ are equal
\end{prop}

\begin{proof}
This is no surprise, but still requires a proof; again compare with \cite[5.5.2]{Ho}. Let $\gamma:  (\C^{\Delta})^{cof} \rightarrow \Ho(\C^{\Delta})$ and $\phi: \C^{cof} \rightarrow \Ho(\C)$ denote the localization functors, where $(\D)^{cof}$ for a model category $\D$ stands for the full subcategory of $\D$ generated by the cofibrant objects (Recall our convention that the homotopy category only contains the cofibrant objects). Then $\gamma(f) = \gamma(g)$ by assumption. Fix a simplicial set $K$. Let $F: \C^{\Delta} \rightarrow \C$  be the functor given by evaluation at $K$, i.e., $F(A) = A \wedge K$. By \cite[5.4.2]{Ho}, $F$ preserves cofibrations and acyclic cofibrations, hence induces a functor on the homotopy categories which we will also denote by $F$. Then $F(\gamma(g)) = F(\gamma(f)$. But by definition, $F\circ \gamma = \phi \circ F$; this means that the two maps $A \wedge K \rightarrow B \wedge K$ induced by $f$ and $g$ are equal in $\Ho(\C)$, and this implies the claim. 
\end{proof}

This means that we can regard $\Ho(Fr(\C))$ as the category of left Quillen functors $\SSet_* \rightarrow \C$, localized at the natural weak equivalences, and thus as a category of derived left Quillen functors $\Ho(\SSet_*) \rightarrow \Ho(\C)$: We send an object $X$ of $Fr(\C)$ to the functor $X \wedge^L -: \Ho(\SSet_*) \rightarrow \Ho(\C)$ and a morphisms $f: X \rightarrow Y$ to the derived natural transformation; the proposition above implies that this is well-defined. 

Now we can prove the main theorem announced at the beginning of this chapter:

\begin{thm}
\label{homotopyframes}
Let $\Ho(Fr(\C))$ be the full subcategory of $\Ho(\C^{\Delta})$ determined by the cosimplicial frames. Then evaluation in degree $0$ $ev_0: \C^{\Delta} \rightarrow \C$ induces an equivalence of categories $\Ho(Fr(\C)) \rightarrow \Ho(\C)$. Furthermore, the suspension functor $\Sigma: \C^{\Delta} \rightarrow \C^{\Delta}$ restricts to a functor $\Sigma: \Ho(Fr(\C)) \rightarrow \Ho(Fr(\C))$ which is an equivalence if $\C$ is stable.
 \end{thm}
 
 \begin{proof}
 First note that $ev_0: \C^{\Delta} \rightarrow \C$ is left Quillen, so we indeed get a functor $ev_0: \C^{\Delta} \rightarrow \C$. \\ 
Let $X$ be an object of $\Ho(\C)$, i.e., a cofibrant object of $\C$. By \ref{existenceframings}, we find a frame $A$ on $X$; this means $ev_0^L(A) \cong X$, which proves essential surjectivity. \\
Let $g: X \rightarrow Y$ be a morphism in $\Ho(\C)$. We may up to isomorphism assume that $X, Y$ are cofibrant-fibrant; then $g$ is represented by an actual morphism $f: X \rightarrow Y$ in $\C$; by \ref{framescover}, we find frames $A$, $B$ on $X$ and $Y$ with a map $F: A \rightarrow B$ covering $f$; now, $ev_0^L(F) = g$. Hence $ev_0^L$ is full. \\
Now let $f, g: A \rightarrow B$ be two maps in $\Ho(\C^{\Delta})$ such that $ev_0^L(F) = ev_0^L(G)$; we may again up to isomorphism assume $A, B$ to be cofibrant-fibrant and that $f, g$ are represented by actual morphisms $F, G: A \rightarrow B$ in $\C^{\Delta}$. Then \ref{uniqueframe} implies that $f  = g$. Hence $ev_0^L$ is faithful; this finishes the proof.
 \end{proof}
 
 \subsection{Frames and the suspension functor} 

Let $\C$ be a model category. We write $(\Sigma, \Omega)$ for the adjoint pair $(- \wedge_S S^1, (-)^{S^1})$ on $\C^{\Delta}$. 

\begin{lem}
\label{sigma2}
For any model category $\C$, the functor $\Sigma: \C^{\Delta} \rightarrow \C^{\Delta}$ is a Quillen functor in the Reedy model structure and preserves cosimplicial frames.
\end{lem}
\begin{proof}
By setting $K = L = S^1$ in Proposition \ref{SM7}, we see that $\Sigma$ preserves cofibrations and acyclic cofibrations.  Furthermore, $\Sigma X$ is the cosimplicial object associated to the functor $X \wedge (S^1 \wedge -): \SSet \rightarrow \C$ which is left Quillen as composition of two left Quillen functors, thus $\Sigma X$ is a cosimplicial frame.
\end{proof}

Unfortunately, even if the underlying model category is stable, $\Sigma$ is not a Quillen equivalence (unless the model category is extremely odd); if we have a weak equivalence $\Sigma X \rightarrow Y$ with $X$ a frame and $Y$ Reedy fibrant, then $\Sigma X$ is again homotopically constant, and so $Y$ is homotopically constant. But looking at the definition of $\Omega Y$, it hardly has a chance to be homotopically constant, so the adjoint map $X \rightarrow \Omega Y$ cannot be a weak equivalence. To remedy this failure, we define another class of "weak equivalences" (which will in general NOT be part of a model structure):
\begin{definition}
A map $f: X \rightarrow Y$ of cosimplicial objects in a model category is a \emph{realization weak equivalence} if for all cosimplicial frames $A$, the induced map $[A,X] \rightarrow [A,Y]$ is an isomorphism in the homotopy category of $\C^{\Delta}$.
\end{definition} 

The definition is made to fit into a potential model structure where the frames are the cofibrant objects; such a model structure exists under the usual conditions which allow localization, see \cite{RSS} and \cite{Dusim}.  For us, it is mainly an auxiliary construction which is helpful to prove Proposition \ref{lifting}.

\begin{lem}
All weak equivalences are realization weak equivalences.
\end{lem}
\begin{proof}
This is clear since weak equivalences $X \rightarrow Y$ induce isomorphisms $[A,X] \rightarrow [A,Y]$ for all $A$.
\end{proof}

\begin{lem}
Let $X$, $Y$ be frames and $f: X \rightarrow Y$ a realization weak equivalence. Then $f$ is a weak equivalence.
\end{lem}
\begin{proof}
We can assume without loss of generality that $X$ and $Y$ are Reedy fibrant. \\
By definition, $f$ induces an isomorphism $[Y,X] \rightarrow [Y,Y]$. Let $g: Y \rightarrow X$ be (a representative of) the preimage of the identity. This means $f \circ g \simeq Id_Y$.
Under the isomorphism $[X,X] \rightarrow [X,Y]$, $g \circ f$ is mapped to $f \circ g \circ f \simeq f$; but since $Id_X $is also mapped to $f$, we must have $g \circ f \simeq Id_X$, hence $f$ is a homotopy equivalence and thus a weak equivalence.
\end{proof}

Realization weak equivalences also have the expected behaviour with respect to the suspension and loop functor:

\begin{prop}
\label{stability}
Let $X$ be a cosimplicial frame over a stable model category $\C$ and $Y$ a Reedy fibrant cosimplicial object. Then a map $f: \Sigma X \rightarrow Y$ is a realization weak equivalence if and only its adjoint $\tilde{f}: X \rightarrow \Omega Y$ is a realization weak equivalence.
\end{prop}

\begin{proof}
Let $A$ be a cosimplicial frame. There is a commutative diagram
\[
\xymatrix{
[A,X] \ar[d]_{\Sigma} \ar[rr]^{[A,\tilde{f}]} & & [A, \Omega Y] \ar[d]_{\cong}\\
[\Sigma A, \Sigma X] \ar[rr]^{[\Sigma A, f]} & & [\Sigma A, Y]
}
\]
The map on the left is an isomorphism since $\C$ is stable, $A$ and $X$ are frames and the homotopy category of frames is equivalent to the homotopy category of $\C$. \\
If $f$ is a realization weak equivalence, the bottom map is an isomorphism since $\Sigma A$ is a frame, so the top map is an isomorphism as well; hence $\tilde{f}$ is a realization weak equivalence. Conversely, if $\tilde{f}$ is a realization weak equivalence, the top map is an isomorphism, hence the lower map also is for any frame $A$. For an arbitrary frame $B$, we can find a frame $A$
 together with a weak equivalence $g: \Sigma A \rightarrow B$ by Theorem \ref{homotopyframes}. In the commutative diagram
 \[
 \xymatrix{
 	[B, \Sigma X] \ar[rr]^{f_*} \ar[d]^{g^*} & & [B, Y] \ar[d]_{g^*}\\
	[\Sigma A, \Sigma X] \ar[rr]^{f_*} & & [\Sigma A, Y]
 }
 \]
 the two vertical maps are isomorphisms since $g$ is a weak equivalence, and the bottom map is an isomorphism. Hence the top map is an isomorphism, proving that $f$ is a realization weak equivalence.  
 \end{proof}

\begin{prop}
\label{lifting}
Let $f: X \rightarrow Y$ be a map of Reedy fibrant cosimplicial objects which is both a realization weak equivalence and a Reedy fibration. Then $f$ has the right lifting property with respect to cosimplicial frames.
 \end{prop}
 
 \begin{proof}
 By definition of a realization weak equivalence and since we have sufficient cofibrancy and fibrancy conditions, each map $A \rightarrow Y$ with $A$ a frame admits a lift up to homotopy $A \rightarrow X$, i.e. an actual map $A \rightarrow X$ making the diagram commutative up to homotopy. It is a standard fact about model categories that in such a triangle, with the right-hand map a fibration, one can change a lift up to homotopy within its homotopy class to an actual lift; see for example in the proof of \cite[6.3.7]{Ho}.
 \end{proof}

\section{Stable frames in arbitrary model categories}

We have already seen how we can describe Quillen pairs $\Sp \rightleftharpoons \C$ for a simplicial model category $\C$. Using $\Sigma$-cospectra, we will extend this description to all model categories. \\ 

\subsection{Stable frames}

We have the following analogue of \ref{stableframing} in arbitrary model categories:
\begin{thm}
\label{genstableframings}
Let $X$ be a $\Sigma$-cospectrum in the model category $\C$. Then the adjoint pair $X \wedge -: Sp \rightleftharpoons \C: \Map(X,-)$ is a Quillen pair if and only if all $X_n$ are cosimplicial frames and the structure maps $\Sigma X_n \rightarrow X_{n-1}$ are weak equivalences.
\end{thm}
\begin{definition}
A $\Sigma$-cospectrum $X$ which is levelwise a frame and has weak equivalences $\Sigma X_n \rightarrow X_{n-1}$ is a \emph{stable frame} on the object $X_{0,0}$.
\end{definition}
\begin{proof} \emph{(of Theorem 6.1)}
For the "only if" direction, note that $X_n$ represents the left adjoint $X \wedge F_n(-): \SSet_* \rightarrow \C$ which is Quillen as composition of two left Quillen functors; hence $X_n$ is a frame. Furthermore, the structure maps $\Sigma X_n \rightarrow X_{n-1}$ is represented by the image under $X \wedge -$ of the natural transformation $F_n(-)  \wedge S^1 \rightarrow F_{n-1}(-)$ which induces a $\pi_*$-isomorphism $F_n(K)  \wedge S^1 \rightarrow F_{n-1}(K)$ for all simplicial sets $K$ and hence is a natural weak equivalence; since this is preserved by left Quillen functors,  $\Sigma X_n \rightarrow X_{n-1}$ represents a natural weak equivalence and hence is a weak equivalence. \\
Conversely, assume all $X_n$ are cosimplicial frames and the maps $\Sigma X_n \rightarrow X_{n-1}$ are weak equivalences. We will use Lemma \ref{detection} to prove that $(X \wedge -, \Map(X,-))$ is a Quillen pair. \\
To check i), let $f: A \rightarrow B$ be an acyclic fibration in $\C$.  We want to see $\Map(X,f): \Map(X,A) \rightarrow \Map(X,B)$ is an acyclic  fibration of spectra, i.e., a level weak equivalence and level fibration. But since $\Map(X,f)_n: \Map(X_n,A) \rightarrow \Map(X_n,B)$ is an acyclic fibration of simplicial sets because $\Map(X_n, -): \C \rightarrow \SSet_*$ is right Quillen for a cosimplicial frame $X_n$, this holds. Note that there are two meanings to $\Map$ here; $\Map(X,-)$ is a functor from $\C$ to Sp since $X$ is a $\Sigma$-cospectrum, and $\Map(X_n, -)$ is a functor from $\C$ to $\SSet$ since $X_n$ is a cosimplicial object.\\
To check ii), let $A$ be a fibrant object in $\C$. Then the simplicial sets $\Map(X_n,A)$ are Kan fibrant since $\Map(X_n,-)$ is right Quillen. The adjoint of the structure maps in $\Map(X,A)$ are the maps $\Map(X_n,A) \rightarrow \Map(\Sigma X_{n-1},A)$. To see that these maps are weak equivalences, it suffices to prove that the functor $\Map(-,A): (\C^{\Delta})^{op} \rightarrow \SSet$ maps weak equivalences between Reedy cofibrant cosimplicial objects to weak equivalences; by Ken Brown's lemma \cite[1.1.12]{Ho}, it is already enough to prove that $\Map(-,A)$ maps acyclic cofibrations between cofibrant objects to acyclic fibrations of simplicial sets.  \\
So let $f: X \rightarrow Y$ be an acyclic cofibration of Reedy cofibrant cosimplicial objects. We want to see that $\Map(f,A)$ is an acyclic fibration. Let $g: K \rightarrow L$ be a cofibration of simplicial sets. Consider the diagram 
\[
\xymatrix{
		K \ar[rrr] \ar[d]_g & & & \Map(Y,A) \ar[d]^{f^*} \\
		L \ar[rrr] & & & \Map(X,A)  }
\]
We need a lift in this diagram to conclude that the map on the right is a trivial fibration.  By playing around with the various adjunctions, one sees that a lift in this diagram is equivalent to a lift in the diagram
\[
\xymatrix{
		K \wedge Y \coprod_{K \wedge X} L \wedge X \ar[rrr]  \ar[d]_{g \diamond f} & & & A \\
		L \wedge Y }
\]
By \cite[5.4.1]{Ho}, the pushout-product on the left is an acyclic cofibration in $\C$; since $A$ is fibrant, we can find a lift. \\
For iii), let $f: A \rightarrow B$ be a fibration between fibrant objects. The map $\Map(X_n, A) \rightarrow \Map(X_n,B)$ is then a fibration since $\Map(X_n,-)$ is right Quillen; hence $\Map(X, A) \rightarrow \Map(X,B)$ is a level fibration. This finishes the proof. 
\end{proof}
\begin{thm}
\label{existence}
Let $\C$ be a stable model category and $A$ a cofibrant-fibrant object of $\C$. Then there is a stable frame $X$ on $A$, i.e., such that $X_{0,0} \cong A$; or equivalently, there is a left Quillen functor $L: Sp \rightarrow \C$ with $L(\U) \cong A$. This frame can be chosen to be fibrant in the model category $\C^{\Delta}(\Sigma)$.
\end{thm}
\begin{proof}
Let $X_0$ be a Reedy fibrant cosimplicial frame on $A$. $\C$ is stable, thus there is a cofibrant-fibrant object $Y$ such that $\Sigma Y \cong A$ in the homotopy category. \\
Choose a Reedy fibrant cosimplicial frame $X_1$ on Y (see \cite[5.2.8]{Ho}), then $X_1 \wedge S^1 \cong \Sigma Y \cong A$ in the homotopy category. Since $X_1 \wedge S^1$ is cofibrant and $A$ is fibrant, this isomorphism in the homotopy category is realized by a weak equivalence $g: X_1 \wedge S^1 \rightarrow A$. This extends to a map $f: \Sigma X_1\rightarrow cA$ where $cA$ is the constant cosimplicial object on $A$. Now the map $X_0 \rightarrow cA$ adjoint to the identity in dimension 0 is a Reedy acyclic fibration since $X_0$ is Reedy fibrant and homotopically constant; thus we may lift $f$ to a map $G: \Sigma X_1 \rightarrow X_0$

The map $G$ is a weak equivalence since it is in level $0$ (the other two maps are weak equivalences in level 0) and both $X_0$ and $\Sigma X_1$ are homotopically constant. Now we just repeat this with $X_1$ and the object $Y$ instead of $X_0$ and $A$ to obtain $X_2$ and a weak equivalence $\Sigma X_2 \rightarrow X_1$; iterating this procedure defines a stable frame $X$. By replacing $X$ fibrantly, which by construction of the factorizations in $\C^{\Delta}{\Sigma}$ and \cite[5.2.8]{Ho} is possible without changing the fibrant object $X_{(0,0)}$, we get a fibrant stable frame, i.e., a stable frame which is pointwise Reedy fibrant and in which all adjoint structure maps $X_n \rightarrow \Omega X_{n-1}$ are fibrations. \\
Alternatively, one can directly refer to Theorem \ref{homotopyframes}, which directly allows one to find a cosimplicial frame $X_1$ with a weak equivalence $\Sigma X_1 \rightarrow X_0$ since $\Sigma$ induces an equivalence on the homotopy category of frames, and iterating this. 
\end{proof}

\begin{rem}
In the case of a simplicial model category, it was possible to choose functorial stable frames if one assumed functorial factorizations. However, the construction just given is highly unfunctorial, and it seems to be impossible to make it functorial. The problem is that the suspension is not necessarily realized by a left Quillen functor. Assuming functorial factorizations, there are always functors $- \wedge S^1: \C \rightarrow \C$ and $\Omega: \C \rightarrow \C$ which have derived functors and represent suspension and loop on the homotopy category. Using these functors, we can find a functorial desuspension of a cofibrant-fibrant object $X$ - we may take a cofibrant replacement $Y$ of $\Omega X$. This yields a weak equivalence $Y \rightarrow \Omega X$; unfortunately, the functors $- \wedge S^1$ and $\Omega(-)$ are not adjoint on the nose, only their derived functors are adjoint. So we cannot use some adjunction isomorphism to produce a weak equivalence $Y \wedge S^1 \rightarrow X$; we only get an isomorphism in the homotopy category $Y \wedge S^1 \cong X$, and there is no canonical way to choose a representative. This makes it plausible that stable frames cannot be constructed functorially.
\end{rem}

For studying derived natural transformations, we will need that two homotopic maps of stable frames induce the same derived natural transformations. For this end, we need some compatibility between the model structures on $\C^\Delta(\Sigma)$, $\C$ and Sp. The next proposition is our equivalent of \cite[5.4.1]{Ho}, and the proof is virtually the same as the one given there.

\begin{prop}
\label{compatible}
Let $\C$ be a model category. Assume $f: X \rightarrow Y$ is a cofibration in $\C^\Delta(\Sigma)$  and $g: A \rightarrow B$ is a cofibration of spectra. Then the induced pushout-product map in $\C$ $f  \Box g: X \wedge B \coprod_{X \wedge A} Y \wedge A \rightarrow Y \wedge B$ is a cofibration which is trivial if $f$ is.
\end{prop}
 
 \begin{proof}
 We may assume that $g$ is one of the generating cofibrations $F_m \partial \Dn \rightarrow F_m \Dn$ using \cite[4.2.4]{Ho}. \\
 In this case, the induced map is the map 
 \[
 X_{m,n} \coprod_{X_m \wedge \partial \Dn_+} Y_m \wedge \partial \Dn_+ \rightarrow Y_{m,n}
 \]
 which is a cofibration by definition of the Reedy cofibrations between cosimplicial objects if $f$ is a cofibration. If $f$ is acyclic, it is also acyclic; see \cite[5.2.5]{Ho}.
 \end{proof}
 
 \begin{cor}
 \label{QQ}
 Let B be a cofibrant spectrum. Then the functor $- \wedge B: \C^\Delta(\Sigma) \rightarrow \C$ preserves cofibrations and acyclic cofibrations and hence has a left derived functor. 
 \end{cor}
 
 \begin{proof}
 We set $A = *$ in \ref{compatible}. The pushout-product map is then just the map $f \wedge B: X \wedge B \rightarrow Y \wedge B$ which is an (acyclic) cofibration if $f$ is. Hence $- \wedge B$ preserves weak equivalences between cofibrant objects and has a left derived functor as claimed.
 \end{proof}
 
 \begin{cor}
 Let $X$, $Y$ be stable frames and $F, G: X \rightarrow Y$ two homotopic maps. Then $F$ and $G$ induce the same derived natural transformations between the derived functors of $X \wedge -$ and $Y \wedge -$. 
 \end{cor}
 
\begin{proof}
Let $A$ be a cofibrant spectrum. We have two maps $F(A), G(A): X \wedge A \rightarrow Y \wedge A$. We claim that these two maps represent the same map in $\Ho(\C)$: Since $- \wedge A$ has a derived functor by the preceding corollary, we get a diagram of functors
\[
\xymatrix{
		\C^\Delta(\Sigma) \ar[d] \ar[rr]^{- \wedge A} & & \C \ar[d] \\
		\Ho(\C^\Delta(\Sigma) \ar[rr]^{- \wedge^L A} & & Ho(\C)
}
\]
which commutes up to a natural isomorphism.  We want to see that $F$ and $G$ go to the same map via the clockwise composition; but since the left vertical map sending them to their homotopy classes already  sends them to the same map, the clockwise composition, being isomorphic to the counterclockwise one, then has to send these two maps to the same map as well.  So for all cofibrant spectra $A$, $F(A), G(A): X \wedge A \rightarrow Y \wedge A$ represent the same map in $\Ho(\C)$, and this is enough to conclude that the derived natural transformations of $F$ and $G$ are equal.
\end{proof}

Unsurprisingly, weak equivalences between stable frames induce natural weak equivalences:

\begin{prop}
Let $f: X \rightarrow Y$ be a weak equivalence of stable frames. Then the derived natural transformation $f \wedge -: X \wedge - \rightarrow Y \wedge -$ is a natural weak equivalence, i.e., a weak equivalence for all cofibrant spectra $A$. 
\end{prop}

\begin{proof}
By \cite[1.3.18]{Ho}, we may as well check the corresponding statement for the right adjoints, i.e., that for a fibrant object $Z$ of $\C$, the map 

\[
\Map(f, Z): \Map(Y, Z) \rightarrow \Map(X,Z)
\]
is a $\pi_*$-isomorphism. We saw above in the proof of Theorem \ref{genstableframings} that the functor 
\[
\Map(-, Z): (\C^{\Delta})^{op} \rightarrow \SSet_*
\]
preserves weak equivalences between frames if $Z$ is fibrant. Since $\Map(Y, Z)$ and $\Map(X,Z)$ are levelwise of the form $\Map(Y_n, Z)$ and $\Map(X_n, Z)$ for frames $X_n$ and $Y_n$ and $f$ is a levelwise weak equivalence, the map
\[
\Map(f, Z): \Map(Y, Z) \rightarrow \Map(X,Z)
\]
is a level weak equivalence. This is what we wanted to prove. Note that $\Map(Y, Z)$ and $\Map(X,Z)$ are $\Omega$-spectra, hence the notions of level weak equivalence and $\pi_*$-isomorphism agree. 
\end{proof}

\begin{definition}
For a stable model category $\C$, let $SF(\C)$ denote the full subcategory of $\C^\Delta(\Sigma)$ given by all stable frames and $\Ho(SF(\C))$ the full subcategory of the homotopy category of $\C^{\Delta}(\Sigma)$ given by stable frames.
\end{definition}

Now we have carried together enough information to prove the stable analogue of  Theorem \ref{homotopyframes}.

\begin{thm}
\label{MAIN}
Let $X$ be a cofibrant object of $\C$, $Y$ a cofibrant-fibrant object; let $\omega X$ be a stable frame on X and $\omega Y$ a fibrant stable frame on Y.  \\
a) Any map $f: X \rightarrow Y$ extends (nonuniquely) to a map $F: \omega X \rightarrow \omega Y$ and hence to a natural transformation $\omega X \wedge - \rightarrow \omega Y \wedge -$ covering $f$ on the sphere spectrum. \\
b) Let $f': X \rightarrow Y$ be homotopic to f. Then any F and F' constructed from f and f' as in a) are homotopic and hence induce the same derived natural transformation between the derived functors of $\omega X \wedge - $ and $\omega Y \wedge -$. \\
c) If f is a weak equivalence, any map $\omega X \rightarrow \omega Y$ as in a) is a natural weak equivalence. \\
d) Evaluation in degree ${(0,0)}$ induces an equivalence of categories
\[
ev_{(0,0)}: \Ho(SF(\C)) \stackrel{\cong}{\rightarrow} \Ho(\C) 
\]
from the homotopy category $\Ho(SF(\C))$ of stable frames in $\C$ to $\Ho(\C)$.
\end{thm}

\begin{proof}
For a), we first extend $f$ to a map $F_0: \omega^0 X \rightarrow \omega^0 Y$: Since $\omega^0 Y$ is Reedy fibrant and homotopically constant, the map $\omega^0 Y \rightarrow cY$ is an acyclic fibration. We also have a map $\omega^0 X \rightarrow cY$ adjoint to $f$, and this map lifts to a map $F_0: \omega^0 X \rightarrow \omega^0 Y$ since $\omega^0 X$ is cofibrant and $\omega^0 Y \rightarrow cY$ is an acyclic fibration. \\
Now we want to produce a map $F_1: \omega^1 X \rightarrow \omega^1 Y$ extending $F_0$ to a map of cospectra up to degree $1$ which is nothing else but a  lift in the diagram
\[
\xymatrix{	
		&  & \omega^1 Y \ar[d] \\		
		\omega^1 X \ar[r] & \Omega \omega^0 X \ar[r]^{F_0} & \Omega \omega^0 Y 
}
\]
 where the maps $\omega^1 X \rightarrow \Omega \omega^0 X$ and $\omega^1 Y \rightarrow \Omega \omega^0 Y$ are the structure maps. Since $\omega Y$ is fibrant, the map on the right is a realization weak equivalence by Proposition \ref{stability} and a Reedy fibration; hence we find a lift in this diagram by Proposition \ref{lifting}. Proceeding like this, we find maps $F_n: \omega^n X \rightarrow \omega^n Y$ which form a morphism of cospectra and cover $f$. This proves a). \\
 For b), it is by the preceding corollary enough to see that the homotopy type of a map $F: \omega X \rightarrow \omega Y$ is determined by the homotopy type of the restriction of $F$ to $f: X \rightarrow Y$. By \cite[5.5.2]{Ho} or Theorem \ref{homotopyframes}, it suffices to see that the homotopy type of $F$ is determined by the homotopy type of $F_0: \omega^0 X \rightarrow \omega^0 Y$ since the homotopy type of $F_0$ is determined by $f$. \\
Let $ev_0: \C^\Delta(\Sigma) \rightarrow \C^{\Delta}$ denote the left Quillen functor given by evaluation in degree $0$. We get an induced map of the simplicial mapping spaces $\Map(\omega X, \omega Y) \rightarrow \Map(\omega^0 X, \omega^0 Y)$. On $\pi_0$, this map induces $[\omega X, \omega Y] \rightarrow [\omega^0 X, \omega^0 Y]$. The maps $F$ and $F'$ go to the same element in $[\omega^0 X, \omega^0 Y]$ by assumption; hence we are finished if we can see that the map $\Map(\omega X, \omega Y) \rightarrow \Map(\omega^0 X, \omega^0 Y)$ is a homotopy equivalence (and thus induces an isomorphism on $\pi_0$). \\
Let $\omega^{\leq n} X$ resp. $\omega^{\leq n} Y$ denote the partial cosimplicial  cospectrum obtained by only taking the first $n+1$ objects of $\omega X$ resp. $\omega Y$. We get a pullback square as follows, where the unnamed maps are induced from the structure maps of $\omega X$ and $\omega Y$ and the mapping spaces are those of $\C^{\Delta}$:
 \[
 \xymatrix{
 		\Map(\omega^{\leq 1} X, \omega^{\leq 1} Y) \ar[d]_{ev_0} \ar[rr]^{ev_1}  & & \Map(\omega^1 X, 		\omega^1 Y) \ar[d] \\
		\Map(\omega^{0}X, \omega^{0}Y) \ar[r]_-{\Omega}  & \Map(\Omega \omega^0 X, \Omega \omega^0 Y) \ar[r] & \Map(\omega^1 X, \Omega \omega^0 Y) \\
 }
 \]

The map on the right is a fibration since the map $\omega^1 Y \rightarrow \Omega \omega^0 Y$ is a fibration and $\Map(\omega^1 X, -)$ is right Quillen; note that this does not follow from \ref{SM7}, but requires an argument that $\omega^1 X \wedge_S -$ preserves acyclic fibrations; compare \cite[5.4.3]{Ho}. By \cite[6.1.2]{Ho}, we have that the induced map $\pi_n \Map(\omega^1 X,  \omega^1 Y) \rightarrow \pi_n \Map(\omega^1 X, \Omega \omega^0 Y)$ is just the map $[\Sigma^n \omega^1 X, \omega^1 Y] \rightarrow [\Sigma^n \omega^1 X, \Omega \omega^0 Y]$  induced by the structure map of $Y$, which is an isomorphism since $\Sigma^n \omega^1 X$ is a frame and $\omega^1 Y \rightarrow \Omega \omega^0 Y$ is a realization weak equivalence. So the map on the right induces an isomorphism on all homotopy groups with basepoint the zero map. To see that it is in fact a $\pi_*$-isomorphism, we have to extend this to all basepoints.  \\
 Since we can find a frame $Z$ with $\Sigma Z \simeq \omega^1 X$, we see that $\Map(\omega^1 X, \omega^1 Y) \simeq \Map(\Sigma Z, \omega^1 Y) \simeq \Omega \Map(Z, \omega^1 Y)$ is a loopspace up to weak equivalence; and similarly $\Map(\omega^1 X, \Omega \omega^0 Y) \simeq \Omega \Map(Z, \Omega \omega^0 Y)$ is a loopspace, and the induced map between the two spaces is up to homotopy  $\Omega$ of the map $\Map(Z, \omega^1 Y) \rightarrow \Map(Z, \Omega \omega^0 Y)$. But in a loopspace, all components are weakly equivalent in a way respected by loop maps, hence we can conclude that $\Map(\omega^1 X, \omega^1 Y) \rightarrow \Map(\omega^1X, \Omega \omega^0 Y)$ is a $\pi_*$-isomorphism and hence an acyclic fibration.  Then the pullback map $\Map(\omega^{\leq 1} X, \omega^{\leq 1} Y) \rightarrow \Map(\omega^0 X, \omega^0 Y)$ is an acyclic fibration as well. \\ 
 Now consider for any $n$ the square
 \[
  \xymatrix{
 		\Map(\omega^{\leq n} X, \omega^{\leq n} Y) \ar[d] \ar[rr]^{ev_n} &  & \Map(\omega^{n} X, 		\omega^{n} Y) \ar[d] \\	
		\Map(\omega^{\leq n-1}X, \omega^{\leq n-1}Y) \ar[r]_{ev_{n-1}} & \Map( \omega^{n-1} X, \omega^{n-1} Y) \ar[r] &  \Map(\omega^{n} X, \Omega \omega^{n-1} Y) \\
 }
 \]
 Again, this is a pullback square and the map on the right is an acyclic fibration, so the map on the left also is. Hence, for any $n$, the map $\Map(\omega^{\leq n} X, \omega^{\leq n}Y) \rightarrow \Map(\omega^{\leq n-1} X, \omega^{\leq n-1}Y)$ forgetting the degree $n$-part is an acyclic fibration. Furthermore, we have $\lim_{n} \Map(\omega^{\leq n} X, \omega^{\leq n}Y)  = \Map(\omega X, \omega Y)$. thus the map $\Map(\omega X, \omega Y) \rightarrow \Map(\omega^0 X, \omega^0 Y)$ is also an acyclic fibration as a limit of acyclic fibrations, proving our claim. \\
 For c), first note that $F_0$ is a weak equivalence since it is a map between homotopically constant cosimplicial objects covering the weak equivalence $f$ in degree $0$. Now we look at the commutative diagram
 \[
 \xymatrix{
 \Sigma \omega X_1 \ar[d] \ar[rr]^{\Sigma F_1} & & \Sigma \omega Y_1 \ar[d] \\
 \omega X_0 \ar[rr]_{F_0} & & \omega Y_0
 }
\]
 The two vertical maps and $F_0$ are weak equivalences, so $\Sigma F_1$ also is. By Theorem \ref{homotopyframes} and since $\omega X_1$ and $\omega Y_1$ are frames, this means $F_1$ is a weak equivalence. By iterating this argument, we find that $F$ is a weak equivalence, which induces a natural weak equivalence between the functors $\omega X \wedge -$ and $\omega Y \wedge -$. \\
 For d), first note that $ev_{(0,0)}$ indeed induces a functor $ev: \Ho(SF(\C)) \rightarrow \Ho(\C)$ since $ev_{(0,0)}: \C^{\Delta}(\Sigma) \rightarrow \C$ is just the functor $- \wedge \U$ which has a derived functor by Corollary \ref{QQ}. 
 Since one can build a stable frame on any cofibrant-fibrant object of $\C$ and every object of $\Ho(\C)$ is isomorphic to such an object, we get that $ev$ is surjective on isomorphism classes of objects. For injectivity, let $X$ and $Y$ be two stable frames such that $ev(X)$ and $ev(Y)$ are isomorphic. We may assume that $X$ and $Y$ are fibrant; then the isomorphism $ev(X) \cong ev(Y)$ is represented by a weak equivalence $f: ev_{(0,0)}(X) \rightarrow ev_{(0,0)} (Y)$. By a), $f$ extends to a map $X \rightarrow Y$, and by c), this map is a weak equivalence. Hence $X$ and $Y$ are isomorphic in the homotopy category of frames. \\
That $ev$ is full is just a reformulation of part a). Given a morphism $g: A \rightarrow B$ in $\Ho(\C)$, we may assume that $A$ and $B$ are cofibrant-fibrant and we obtain an actual morphism $f: A \rightarrow B$ in $\C$. Let $X$ be a stable frame on $A$, $Y$ a fibrant stable frame on $X$. By a), $f$ extends to a map $F: X \rightarrow Y$ and $ev_{(0,0)}(F) = f$, hence $ev(F) = g$. \\
Finally, that $ev$ is faithful follows directly from part c) since two maps of stable frames which are homotopic in degree $(0,0)$ are homotopic.
\end{proof}

\begin{cor}
\label{uniquenessgen}
Let $L$, $L'$ be two left Quillen functors $\Sp \rightarrow \C$ with a weak equivalence $L(\U) \simeq L'(\U)$. Then there is a canonical isomorphism between the derived functors of $L$ and $L'$ covering the given isomorphism on the sphere spectrum.
\end{cor}

\begin{proof}
Let $K$ and $K'$ be the stable frames associated to $L$ and $L'$. By replacing $K'$ fibrantly, we get a fibrant stable frame $Z$ which is weakly equivalent to $K'$, and we have a weak equivalence $L(\U) \rightarrow L'(\U) \rightarrow Z \wedge \U$. By the preceding theorem, this weak equivalence extends to a natural weak equivalence $K \rightarrow Z$ whose derived natural transformation is a canonical isomorphism, and we also have a natural weak equivalence $K' \rightarrow Z$ obtained from the fibrant replacement. These two natural weak equivalences yield the desired isomorphism.
\end{proof}

\begin{rem}
We may consider stable frames and the category $\Ho(SF(\C))$ for any model category, though there may be not so many stable frames: For example, in topological spaces, no non-contractible object is an infinite suspension, so we can only build stable frames on contractible objects and consequently, $\Ho(SF(\Top))$ is equivalent to the trivial category. A model category is stable if and only if evaluation at the sphere spectrum induces an equivalence of categories $\Ho(SF(\C)) \rightarrow \Ho(\C)$.
\end{rem}

\begin{rem}
\label{othermodels}
The "universal properties" of $\Sp$ formulated in \ref{existence} and \ref{MAIN} certainly depend on the concrete structure of $\Sp$, and not any model of $\SHC$ will share these properties, though conversely $\Sp$ is determined up to Quillen equivalence by them. It is not, however, determined up to equivalence; we could have built a category of $S^2$-spectra in the sense of Chapter 5 which carry a stable model structure, and this category of spectra shares the universal properties - just use $\Sigma^2$-cospectra, and all proofs in this section go through virtually unchanged. In fact, one might take a sequence $a_0, a_1, \dots$ of positive natural numbers and define a category of spectra with structure maps $X_n \wedge S^{a_n} \rightarrow X_{n+1}$ which carries a level model structure, then localize at the maps $F_n S^{a_{n-1}} \rightarrow F_{n-1} S^0$ to obtain a stable model category which will also fulfill \ref{existence} and \ref{MAIN}.
\end{rem}

\section{Enrichments}

We adopt the definitions of (closed) modules over a monoidal category from \cite[4.1]{Ho}: this is how the homotopy category of a stable model category will be enriched over $\SHC$.

Now we construct our enrichment functor (or, rather, module functor). Let $\C$ be a stable model category. We have a functor
\[
- \wedge -: SF(\C) \times \Sp \rightarrow \C
\]

\begin{lem}
This functor has a derived functor $\Phi: \Ho(SF(\C)) \times \SHC \rightarrow \Ho(\C)$
\end{lem}

Note that there is no claim that $- \wedge -$ is a Quillen bifunctor; we just want an induced functor on the homotopy categories.

\begin{proof}
Let $f: X \rightarrow Y$ be a weak equivalence of stable frames and $g: A \rightarrow B$ a weak equivalence between cofibrant spectra. We have to check that the map $f \wedge g: X \wedge A \rightarrow Y \wedge A \rightarrow Y \wedge B$ is a weak equivalence. But since $Y \wedge -$ is left Quillen, it preserves weak equivalences between cofibrant objects, so the second map is a weak equivalence; and since $f$ induces a natural weak equivalence and $A$ is cofibrant, the first map also is a weak equivalence.
Hence the composite functor
\[
\xymatrix{
SF(\C) \times \Sp \ar[r] & \C \ar[r] & \Ho(\C)
}
\]
sends weak equivalences between cofibrant objects to isomorphisms. So we have a derived functor $\Phi: \Ho(SF(\C)) \times \SHC \rightarrow \Ho(\C)$
\end{proof}

Now choose an inverse $\omega$ to the equivalence $ev_{\U}: \Ho(SF(\C)) \rightarrow \Ho(\C)$. In the terminology of \cite{Ho}, $\omega$ is what one might call a \emph{stable framing} for $\C$; the choice of $\omega$ boils down to choosing, for each cofibrant object $A$ of $\C$, a fibrant stable frame $\omega A$ with a weak equivalence $A \rightarrow \omega A \wedge \U \cong (\omega A)_{0,0}$. The following definition depends on the choice of $\omega$, but in no essential way. \\
Now we can define the enrichment functor:

\begin{definition}
The enrichment functor
\[
\otimes: \Ho(\C) \times \SHC \rightarrow \Ho(\C)
\]
is given as the composition
\[
\xymatrix{
\Ho(\C) \times \SHC \ar[rr]^-{\omega \times Id} & & \Ho(SF(\C)) \times \SHC \ar[rr]^-{\Phi} & & \Ho(\C) 
}
\]
\end{definition}

Having an object $X$ of $\C$ and a cofibrant spectrum $A$, we thus obtain $X \otimes A$ by taking the fibrant stable frame $\omega X$ and computing $\omega X \wedge A$ which is a model for $X \otimes A$. Maps of spectra are just plugged into the functor $\omega X \wedge -$, and a map $X \rightarrow Y$ in $\C$ can be extended to a map $\omega X \rightarrow \omega Y$ which gives a map $\omega X \wedge A \rightarrow \omega Y \wedge A$. Also note that for a map $A \rightarrow B$ of spectra, the diagram
\[
\xymatrix{
\omega X \wedge A \ar[d] \ar[rr] & & \omega X \wedge B \ar[d] \\
\omega Y \wedge B \ar[rr] & & \omega Y \wedge B 
}
\]
is commutative since $\omega X \rightarrow \omega Y$ induces a natural transformation; hence our description really gives rise to a functor. \\

\begin{thm}
For $\C = \Sp$, the functor $\otimes: \SHC \times \SHC \rightarrow \SHC$ makes $\SHC$ into a monoidal category with unit $\U$. For arbitrary $\C$, $\otimes: \Ho(\C) \times \SHC \rightarrow \Ho(\C)$ makes $\Ho(\C)$ into a closed $\SHC$-module with respect to the monoidal structure on $\SHC$ given by $\otimes$. A left Qullen functor $\C \rightarrow \D$ between stable model categories induces an $\SHC$-module functor $\Ho(\C) \rightarrow \Ho(\D)$.
\end{thm}

\begin{proof}
Consider the categories $\Ho(SF(\C))$ and $\Ho(SF(\Sp))$, regarded as functor categories. We obtain a functor $\Ho(SF(\C)) \times \Ho(SF(\Sp)) \rightarrow \Ho(SF(\C))$ by composition of derived Quillen functors $\SHC \rightarrow \SHC$ with derived Quillen functors $\SHC \rightarrow \Ho(\C)$ and horizontal composition of natural transformations.  Because of our conventions regarding left Quillen functors, the composition of their derived functors is strictiy associative, and the identity of $\SHC$, which is a left derived functor, acts as strict identity on $\Ho(SF(\C))$. Hence $\Ho(SF(\Sp))$ is monoidal and $\Ho(SF(\C))$ is a $\Ho(SF(\Sp))$-module. Clearly, composition with derived left Quillen functors induces (strict) $\Ho(SF(\Sp))$-module functors $\Ho(SF(\C)) \rightarrow \Ho(SF(\D))$. \\
Since $\SHC$ is equivalent to $\Ho(SF(\Sp))$, we can, after choosing inverse equivalences, pull the monoidal structure from the latter category over to $\SHC$; this destroys strict associativity and strict unitality, but it is still a coherent monoidal product, and this is actually the definition we have given above. That the result is again a monoidal category is certainly no surprise; the proof is just a long, tedious and uninspired diagram chase. Also observe that for a left Quillen functor $F: \C \rightarrow \D$, the following diagram commutes
\[
\xymatrix{
\Ho(SF(\C)) \ar[rr]^{F} \ar[d]_{ev_{\U}} & & \Ho(SF(\D)) \ar[d]^{ev_{\U}} \\
\Ho(\C) \ar[rr]^{F} & & \Ho(\D)
}
\]
Again, by tedious diagram chases, one verifies that the lower functor is an $\SHC$-module functor since the upper one is. This also implies that the choice of inverse $\omega$ is immaterial: The identity of $\C$, with maybe two different choices of $\omega$, is an equivalence of $\SHC$-modules.  That these modules are closed is clear since derived left Quillen functors have right adjoints. This proves the theorem.
\end{proof}


In particular, we have constructed a smash product on $\SHC$; the obvious question is whether we have actually constructed something new. The following Theorem says that this is not so.

\begin{thm}
\label{monoid}
Let $\C$ be a monoidal stable model category with pairing $\Box$ and unit $U$. Choose a left Quillen functor $F: \Sp \rightarrow \C$ sending $\U$ to a cofibrant replacement of $U$. Then the following holds:
\begin{itemize}
\item The composition 
\[
\xymatrix{
\Ho(\C) \times \SHC \ar[rr]^{Id \times F} & & \Ho(\C) \times \Ho(\C) \ar[rr]^{- \Box -} & & \Ho(\C)
}
\]
is a possible model for the enrichment functor for $\C$.
\item The derived functor $F: \SHC \rightarrow \Ho(\C)$ is strong monoidal.  In particular, if $\C$ is any symmetric monoidal model for stable homotopy theory, then $F$ induces a strong monoidal equivalence.
\end{itemize}
\end{thm}

At first glance, this seems to be an extremely strong statement, but it actually is not, and it is already mainly known: In \cite{Smonoidal}, it is proven that symmetric spectra are in a certain sense initial among all monoidal model categories, and an analogue of our theorem holds for $\Ho(\Sp^{\Sigma})$ instead of $\Ho(\Sp)$. Hence the only new statement we make is that our smash product on $\SHC$ is compatible with the one from $\Sp^{\Sigma}$.

\begin{proof}
For i), we construct a particular inverse $\omega: \Ho(\C) \rightarrow \Ho(SF(\C))$ to evaluation at the sphere spectrum as follows: For any cofibrant object $X$ of $\C$, the functor $X \Box -$ is left Quillen, and we set $\omega (X) = X \Box F(-)$, where by abuse of notation we also denote the derived product on $\Ho(\C)$ by $\Box$.  By definition, $\omega(X) (\U) \cong X$ in the homotopy category. A morphism $f: X \rightarrow Y$ in $\C$ induces a natural transformation $\omega X \rightarrow \omega Y$ covering, up to homotopy, $f$ on the sphere spectrum; on the homotopy category, these constructions hence yield a functor $\omega: \Ho(\C) \rightarrow \Ho(SF(\C))$ inverse to evaluation at the sphere spectrum, and we may use this particular inverse to construct the enrichment. Now, for an object $X$ of $\Ho(\C)$ and a spectrum $A$, we have $X \otimes A = \omega(X)(A) = X \Box F(A)$ by definition, and the claim follows. \\
For ii), we have to produce a natural isomorphism $F(A \otimes B) \cong F(A) \Box F(B)$. By naturality of the enrichment with respect to left Quillen functors, we have an isomorphism $F(A \otimes B) \cong F(A) \otimes B$, and by part i) we may arrange things such that $F(A) \otimes B = F(A) \Box F(B)$. This isomorphism is natural in both variables; it remains to check the commutativity of various coherence diagrams, cf. \cite[XI.2]{Maclan}. \\
The first coherence diagram is
\[
\xymatrix{	
	F(A) \Box(F(B) \Box F(C))  \ar[rr] \ar[d] & & (F(A) \Box F(B)) \Box F(C) \ar[d] \\
	F(A) \Box F(B \otimes C) \ar[d] & & F(A \otimes B) \Box F(C) \ar[d] \\
	F(A \otimes (B \otimes C)) \ar[rr] & & F((A \otimes B) \otimes C)
}
\]
for spectra $A, B, C$. Using the equality $F(A) \Box F(B) = F(A) \otimes B$, this diagram reads

\[
\xymatrix{	
	F(A) \Box (F(B) \otimes C)  \ar[rr] \ar[d] & & (F(A) \otimes B) \otimes C \ar[d] \\
	F(A) \otimes (B \otimes C) \ar[d] \ar[urr] & & F(A \otimes B) \otimes C \ar[d] \\
	F(A \otimes (B \otimes C)) \ar[rr] & & F((A \otimes B) \otimes C)
}
\] 
The diagonal arrow is obtained from associativity of the enrichment. Now, the lower pentagram is one of the coherence diagrams required in the definition of an $\SHC$-module functor and hence commutes; it remains to see commutativity of the upper triangle. First note this is not a triviality; we have once used associativity of $\Box$ and once associativity of $\otimes$, and there is no immediate reason these should be related - indeed, we have not yet used that $\Box$ is coherent. By letting $C$ vary, we obtain derived left Quillen functors $F(A) \Box (F(B) \otimes -)$, $(F(A) \otimes B) \otimes -$ and $F(A) \otimes (B \otimes -)$ and natural transformations between these; since the behaviour of derived Quillen functors is determined by what is happening on the sphere spectrum by \ref{MAIN}, we may reduce to $C = \U$. \\
In this case, coherence of the $\SHC$-action tells us that the composition $F(A) \otimes (B \otimes \U) \cong F(A) \otimes B \cong (F(A) \otimes B) \otimes \U$ of the two unit maps is the associativity isomorphism for $A, B, \U$. The triangle hence takes the following form, again using $X \otimes A = X \Box F(A)$:
\[
\xymatrix{
		F(A) \Box (F(B) \Box F(\U)) \ar[rr] \ar[d] & & (F(A) \otimes B) \otimes \U \\
		F(A) \otimes (B \otimes \U) \ar[rr] & & F(A) \otimes B \ar[u]
		}
\]
Also, by construction, we have not only an equality $X \otimes \U = X \Box F(\U)$, but also the two possible unit maps $X \otimes \U \cong X$ from the $\SHC$-action and $X \otimes F(\U) \cong X \otimes U \cong X$ from the $\Box$-product are the same, so we do not have to worry about using the wrong unit map. Translating back in the $\Box$-language, we hence may as well consider the diagram
\[
\xymatrix{
		F(A) \Box (F(B) \Box F(\U)) \ar[rr] \ar[d] & & (F(A) \Box F(B)) \Box  F(\U) \\
		F(A) \Box F(B \otimes \U) \ar[rr] & & F(A) \Box F(B) \ar[u]
		}
\]
The lower horizontal map is the map obtained by going back up to $F(A) \Box (F(B) \Box F(\U))$ and then using the unit map to $F(A) \Box F(B)$, and the top horizontal map is the associativity isomorphism in $\Ho(\C)$. This is now a coherence diagram for $\Box$ and hence commutes, proving commutativity of the first coherence diagram. \\
The other coherence diagram to check is the one involving the right unit isomorphisms:
\[
\xymatrix{
		U \Box F(B) \ar[rr] \ar[d] & & F(B) \\
		F(\U) \Box F(B) \ar[rr]  & & F(\U \otimes B) \ar[u] }
\]
As above, we may reduce to $B = \U$, where the commutativity is clear from the properties of $F$. 
\end{proof}

\begin{cor}
Let $\C$ be any symmetric monoidal model for the stable homotopy category. The smash product on $\Ho(\Sp)$ we have constructed is, under the Quillen equivalence $\Sp \rightarrow \C$ sending $\U$ to the unit of $\C$, equivalent to the one in $\Ho(\C)$.
\end{cor}

\begin{rem}
The proof of the Theorem actually proves something slightly stronger: There is no need for a monoidal structure on $\C$, one only needs a monoidal structure on $\Ho(\C)$ which is induced by Quillen functors and natural transformations, but which may be not associative or unital on the nose - as for example our construction of the smash product in $\SHC = \Ho(\Sp)$. 
\end{rem}

\begin{example}
Let $R$ be any ring. There is a stable model structure on the category $Ch(R)$ of unbounded chain complexes of $R$-modules with the homology isomorphisms as weak equivalences, see \cite[2.3.3]{Ho}. This is actually a monoidal model category under the tensor product of chain complexes, see \cite[4.2.13]{Ho}. For any $R$-module $M$, we denote the chain complex which has a copy of $M$ in degree $0$ and $0$ in all other degree also by $M$. For any chain complex $X$, we have that $H_*(X) \cong [R, X]_*$ where $[-,-]$ denotes morphisms in the derived category; since all chain complexes are fibrant and $R$ is cofibrant, $[R, X]_*$ is just given by chain homotopy classes of morphism $R \rightarrow X$. \\
 We can find a left Quillen functor $F: \Sp \rightarrow Ch(R)$ sending $\U$ up to weak equivalence to $R$. By the preceding theorem, the $\SHC$-enrichment of $\Ho(Ch(R))$ is given by $X \otimes F(A)$ for a spectrum $A$ and a chain complex $X$, so we would like to understand what kind of chain complex $F(A)$ is. Since $F: \SHC \rightarrow \Ho(Ch(R))$ is a triangulated functor and homology is homomorphisms out of $R$ in $Ch(R)$, it follows that the functor $h_*(F(-))$ is a homology theory, and the coefficient groups are easily calculated to be $R$ in degree $0$ and $0$ else. Hence the homology of $F(A)$ is the ordinary singular homology of $A$ with $R$-coefficients and $F(A)$ is a kind of "singular chain complex" of the spectrum $A$. This also gives an easy proof of the K\"unneth formula for spectra: We know that $F$ is monoidal, hence we must have $F(X \wedge Y) \cong F(X) \otimes F(Y)$! \\
Also note that $\Map(R,R)$ is an Eilenberg-MacLane spectrum for $R$ since 
\[
[\U, \Map(R,R)]_* \cong [R, R]
\]
In general, given two objects $X, Y$ of a stable model category, we can form a cohomology theory by taking the derived Quillen functor $G: \SHC \rightarrow \Ho(\C)$ sending $\U$ to $Y$ and then setting $h^*(A) = [G(A), X]_* \cong [A, \Map(Y,X)]_*$ with coefficients $[Y, X]_*$.
\end{example}

\subsection{Compatibility with the triangulated structure}

Both $\SHC$ and $\Ho(\C)$ are triangulated categories; so the enrichment functor $\SHC \times \Ho(\C) \rightarrow \Ho(\C)$ ought to be compatible with this structure.  The following theorem tells us this is indeed so:

\begin{thm}
\label{biexact}
The functor $\otimes: \Ho(\C) \times \SHC \rightarrow \Ho(\C)$ is biexact: For any object $A$ of $\SHC$ and any $X$ in $Ho(\C)$, the functors
\[
- \otimes A: \Ho(\C) \rightarrow \Ho(\C)
\]
and
\[
X \otimes -: \SHC \rightarrow \Ho(\C)
\]
preserve triangles and are additive.
\end{thm}

\begin{proof} 
The statement involving $X$ holds since $X \otimes -$ is the derived functor of the Quillen functor $X \wedge -$ which respects the triangulated structure by \cite[6.4.1]{Ho}. The other direction needs a careful analysis of what happens with a triangle under $- \otimes A$ for which we need more machinery; we defer the main part of the proof to the end of this section.
\end{proof}

The idea of the proof is to start with a cofiber sequence $X \rightarrow Y \rightarrow Z$ and try to build a sequence of stable frames $\omega X \rightarrow \omega Y \rightarrow \omega Z$ which is again a cofiber sequence, and then plugging $A$ into the three functors to obtain a cofiber sequence $X \wedge A \rightarrow Y \wedge A \rightarrow Z \wedge A$. To carry through this program, we first need the following useful proposition which is from \cite[5.2.6]{Ho}.

\begin{prop}
\label{cube}
Let $\C$ be a model category. Suppose we have two pushout squares $X_0$, $X_1$ of the form
\[
\xymatrix{
	A_i \ar[rr]^{f_i} \ar[d]_{g_i} & & B_i \ar[d] \\
	C_i \ar[rr] & & D_i
}
\]
for $i = 0, 1$ such that $f_0$ and $f_1$ are cofibrations and all objects are cofibrant. Furthermore assume we have a map of diagrams $X_0 \rightarrow X_1$ such that the maps $A_0 \rightarrow A_1$, $B_0 \rightarrow B_1$ and $C_0 \rightarrow C_1$ are weak equivalences. Then the map on the pushouts $D_0 \rightarrow D_1$ is also a weak equivalence.
\end{prop}

\begin{proof}
See \cite[5.2.6]{Ho}. 
\end{proof}

We can use this to prove that the cofiber of a cofibration between stable frames is again a stable frame:

\begin{prop}
Let $f: \omega X \rightarrow \omega Y$ be a cofibration (i.e., a pointwise Reedy cofibration) between two stable frames. Then the mapping cone $C(f)$ is again a stable frame.
\end{prop}

\begin{proof}
We have to check that $C(f)$ is pointwise a cosimplicial frame and that the structure maps are weak equivalences. Since colimits are taken componentwise, we have to check that the cofiber of a Reedy cofibration $g: A \rightarrow B$ between two cosimplicial frames is again a frame. The object $C(g)$ is Reedy cofibrant since $\cdot \rightarrow C(g)$ is the pushout of the cofibration $g$, so it remains to check that it is also homotopically constant. A cosimplicial frame is homotopically constant if and only if the canonical map $A \rightarrow c(A_0)$ into the constant cosimplicial object on $A_0$ is a weak equivalence. Since colimits of cosimplicial objects are also taken componentwise, we have the two pushout diagrams 

\[
\xymatrix{
		A \ar[rr]^{g} \ar[d] & & B \ar[d]  & & & c(A_0) \ar[rr]^{c(g_0)} \ar[d] & & c(B_0) \ar[d] \\
		\cdot \ar[rr] & & C(g) & & & \cdot \ar[rr] & & c(C(g)_0)
}
\]

The canonical maps from a cosimplicial object $Y$ to $c(Y_0)$ define a map of diagrams between these two pushout diagrams. Since the maps $A \rightarrow c(A_0)$ and $B \rightarrow c(B_0)$ are weak equivalences and the top maps in the two diagrams are cofibrations, we can apply Proposition \ref{cube} to conclude that $C(g)$ is homotopically constant. \\
To check that the structure maps in $C(f)$ are weak equivalences, note that we have a pushout square
\[
\xymatrix{
	\Sigma \omega^n X \ar[rr] \ar[d] & & \Sigma \omega^n Y \ar[d] \\
	\cdot \ar[rr] & & \Sigma C(f)_n
}
\] 
since $\Sigma$ is a left adjoint; and we have a pushout square
\[
\xymatrix{
	\omega^{n-1} X \ar[rr] \ar[d]  & & \omega^{n-1} Y \ar[d] \\
	\cdot \ar[rr] & & C(f)_{n-1}
}
\]
The structure maps of $\omega X$ and $\omega Y$ induce a map of diagrams between these two diagrams, and the induced map on the pushouts is the structure map of $C(f)$. Since the maps $\Sigma \omega^n X \rightarrow \omega^{n-1} X$ and $\Sigma \omega^n Y \rightarrow \omega^{n-1} Y$ are weak equivalences, so is the structure map of $C(f)$ by Proposition \ref{cube}.
\end{proof}

With this in hand, we can realize a cofiber sequence in $\C$ by a cofiber sequence of stable frames, at least up to weak equivalence:

\begin{prop}
Let $f: X \rightarrow Y$ be a cofibration between cofibrant objects in a stable model category $\C$. Then we can find stable frames $\omega X$, $\omega' Y$ with weak equivalences $X \rightarrow \omega X \wedge \U$  and $Y \rightarrow \omega' Y \wedge \U$ and a Reedy cofibration $F: \omega X \rightarrow \omega' Y$ which extends $f$, i.e., such that the diagram
\[
\xymatrix{
	X \ar[d]^{\simeq} \ar[rr]^{f} & & Y \ar[d]^{\simeq} \\
	\omega X \wedge \U \ar[rr]^{F \wedge \U} & & \omega' Y \wedge \U \\
}
\] 
commutes. The cofiber $C(F)$ of $F$ is then a stable frame on an object weakly equivalent to the cofiber of $f$.
\end{prop}

\begin{proof}
By replacing fibrantly, we obtain a map $Rf: RX \rightarrow R'Y$. By factoring this into a cofibration followed by an acyclic fibration, we obtain a cofibration $g: RX \rightarrow RY$ and an acyclic fibration $p: RY \rightarrow R'Y$. Then we can choose a lift in the diagram
\[
\xymatrix{
		X \ar[r]  \ar[d]_{f} & RX \ar[r]^{g} & RY \ar[d] \\
		Y \ar[rr] & & R'Y
}
\]
and we obtain a commutative diagram
\[
\xymatrix{
		X \ar[rr]^{f} \ar[d] & & Y \ar[d] \\
		RX \ar[rr]^{g} & & RY 
}
\]
with the vertical maps weak equivalences and the horizontal ones cofibrations. \\
Now we can choose fibrant stable frames $\omega X$ and $\omega' Y$ on $RX$ and $RY$ respectively and a map $F': \omega X \rightarrow \omega Y$ extending $g$. It is then possible by construction of the factorizations in $\C^{\Delta}(\Sigma)$ and \cite[5.2.8]{Ho} to factor $F'$ as a cofibration $F: \omega X \rightarrow \omega Y$ followed by an acyclic fibration $p: \omega Y \rightarrow \omega' Y$ such that $p$ is the identity in degree $(0,0)$ and $F_{(0,0)}$ is just $g$. This $F$ satisfies all conditions required in the proposition. Furthermore, we have $C(F)_{(0,0)} = C(F_{(0,0)}) \cong C(g)$ and the weak equivalences $X \rightarrow RX$ and $Y \rightarrow RY$ induce a weak equivalence $C(f) \rightarrow C(g)$ by Proposition \ref{cube}, so $C(F)$ is, up to weak equivalence, a stable frame on the cone of $f$ as desired.
\end{proof}

\begin{prop}
Let $X \rightarrow Y \rightarrow Z$ be a cofiber sequence in the stable model category $\C$ and $A$ a cofibrant spectrum. Then we can find a model in $\C$ for the sequence $X \otimes A \rightarrow Y \otimes A \rightarrow Z \otimes A$ in $\Ho(\C)$ which is again a cofiber sequence in $\C$. 
\end{prop}

\begin{proof}
Choose stable frames $\omega X$, $\omega Y$ and a cofibration $\omega X \rightarrow \omega Y$ as in the preceding proposition. We obtain a cofiber sequence $\omega X \rightarrow \omega Y \rightarrow C(F)$ which is, up to weak equivalence, our given cofiber sequence in degree $(0,0)$. Since evaluation at $A$ is defined as a colimit, the functor $- \wedge A: \C^{\Delta}(\Sigma) \rightarrow \C$ preserves colimits, and it preserves cofibrations by Corollary \ref{QQ}. 
Thus the sequence $\omega X \wedge A \rightarrow \omega Y \wedge A \rightarrow C(F) \wedge A$ is again a cofiber sequence, and it represents the sequence $X \otimes A \rightarrow Y \otimes A \rightarrow Z \otimes A$. 
\end{proof}

Now we can prove Theorem \ref{biexact}. First note that $\omega X \wedge A$ is defined as a colimit, so $- \wedge A$ commutes with colimits, in particular with coproducts. Since the coproduct in the homotopy category is the derived functor of the coproduct in the model category, this implies that $- \otimes A$ commutes with coproducts, hence it is an additive functor. \\
To see that $- \otimes A$ commutes with the suspension, recall that $\C^{\Delta}{\Sigma}$ is simplicial and denote the simplicial smash product by $\wedge_S$. The object $\omega X \wedge_S \Delta[1]$ is a cylinder object for $\omega X$ because this is true levelwise and cofibrations and weak equivalences of cospectra are defined levelwise. Since $- \wedge A$ preserves cofibrations and weak equivalences between cofibrant objects and commutes with coproducts, $(\omega X \wedge_S \Delta[1]) \wedge A$ is a cylinder object for $\omega X \wedge A$. Now consider the diagram
\[
\xymatrix{
		((\omega X \wedge_S \Delta[0]) \coprod (\omega X \wedge_S \Delta[0])) \wedge A \ar[rrr]^-{\cong} \ar[d] & & & \omega X \wedge A \coprod \omega X \wedge A \ar[d] \\
		(\omega X \wedge_S \Delta[1]) \wedge A \ar[rrr]^-{=} \ar[d] & & &  (\omega X \wedge_S \Delta[1]) \wedge A \ar[d] \\
		(\Sigma \omega X) \wedge A \ar[rrr] & & & Z
}
\]
where $Z$ is the cofiber of the upper right vertical map; $Z$ is a model for the suspension of $\omega X \wedge A$ which represents $X \otimes A$; hence $Z \cong \Sigma X \otimes A$ in $Ho(\C)$. Since $- \wedge A$ commutes with taking cofibers, the lower horizontal map is an isomorphism; and $\Sigma \omega X$ is a stable frame on a model for the suspension of $X$, hence $(\Sigma \omega X) \wedge A$ is a model for $(\Sigma X)\otimes A$. So we get an isomorphism $(\Sigma X) \otimes A \cong \Sigma(X \otimes A)$ which is by construction also natural in $X$. \\
Now we prove that $- \otimes A$ preserves triangles. By abuse of notation, we also denote models in the model category itself for a suspension by $\Sigma$. Let $X \stackrel{f}{\rightarrow} Y \rightarrow Z \rightarrow \Sigma X$ be a triangle in $\Ho(\C)$; we may assume that $X \rightarrow Y \rightarrow Z$ is an actual cofiber sequence. Choose a cofiber sequence of stable frames $\omega X \stackrel{F}{\rightarrow} \omega Y \rightarrow \omega Z$ as in the proposition above. By passing to the sequence $(\omega X)_{(0,0)} \rightarrow (\omega Y)_{(0,0)} \rightarrow (\omega Z)_{(0,0)}$ if necessary, we can assume that this sequence of stable frames has our original cofiber sequence sitting in degree $(0,0)$. Choose a cone $C(\omega X)$ on $X$; in degree $(0,0)$, this is a cone on $X$. Forming the mapping cone $\omega Y \coprod_F C(\omega X)$, we see by the pointwise definition of colimits of cospectra that this has $Y \coprod_f C(\omega X)_{(0,0)}$ in degree $(0,0)$. Contracting $\omega Y$ to a point in this mapping cone has in degree $(0,0)$ the effect of contracting $Y$ to a point, so we see that the connecting map $\omega Z \simeq \omega Y \coprod_F C(\omega X) \rightarrow \Sigma \omega X$ covers the connecting map $Z \simeq Y \coprod_f C(\omega X)_{(0,0)} \rightarrow \Sigma X$. Another way to express this is that the equivalence $ev_{(0,0)}: \Ho(SF(\C)) \rightarrow \Ho(\C)$ is actually an equivalence of triangulated categories. \\
Since $- \wedge A$ commutes with colimits, we have $(\omega Y \coprod_F C(\omega X)) \wedge A \cong (\omega Y \wedge A) \coprod_{F \wedge A} (C(\omega X) \wedge A)$. Contracting $\omega Y$ respectively $\omega Y \wedge A$, we obtain a diagram
\[
\xymatrix{
	 (\omega Y \coprod_F C(\omega X)) \wedge A \ar[rrr] \ar[dd]_{\cong} & & &  (\Sigma \omega X) \wedge A \ar[dd]^{\cong} \\
	 \\
	 (\omega Y \wedge A) \coprod_{F \wedge A} (C(\omega X) \wedge A) \ar[rrr] & & & \Sigma (\omega X \wedge A) 	 
	 }
\] 
where the right vertical map was the isomorphism $(\Sigma X) \otimes A \cong \Sigma(X \otimes A)$ constructed above.  Since $- \wedge A$ commutes with colimits, this commutes.  But this exactly means that the connecting map for the triangle $X \rightarrow Y \rightarrow Z \rightarrow \Sigma X$ is sent to the connecting map of the triangle $X \otimes A \rightarrow Y \otimes A \rightarrow Z \otimes A \rightarrow \Sigma(X \otimes A)$, and this was the only thing missing to conclude that $- \otimes A$ is triangulated.

\section{The smash product on $\SHC$}

In this chapter, we want to give an explicit description of the smash product on $\SHC = \Ho(\Sp)$ as we have constructed it above. Then we will review the original definition of the smash product as it is given in \cite{adams}; it then easily follows that this smash product is the same one as ours.

\subsection{Quillen endofunctors of spectra}

Given two spectra $A$ and $B$, the following is the most naive candidate for $A \wedge B$: Choose a function $q: \N \rightarrow \N$ which is monotone, $q(n) \leq n$ and such that $q(n+1)-q(n)$ is at most 1. Then $p = Id-q: \N \rightarrow \N$ has the same properties and $p+q = Id$. Furthermore, we demand that both $p$ and $q$ are unbounded. Then for spectra $A$ and $B$, we define the \emph{naive smash product with respect to q} $A \wedge_q B$ levelwise as
\[
(A \wedge_q B)_n = A_{q(n)} \wedge B_{p(n)}
\]
with the following structure maps: If $q(n+1) = q(n)$, we use the structure map of $B$ to obtain a map $A_{q(n)} \wedge B_{p(n)} \wedge S^1 \rightarrow A_{q(n+1)} \wedge B_{p(n+1)}$; else we use the structure map of $A$ after commuting the $S^1$ past the $B_{p(n)}$.  Clearly, this is a functorial construction - both $A \wedge_q -$ and $- \wedge_q B$ are functors $\Sp \rightarrow \Sp$.

\begin{rem}
One might want to make up for the "commuting the $S^1$ past the $B$ factor" in some way, and this in in fact appers in the original definition in topological spaces; however, there is no natural way to do this in our simplicial context; and all our arguments go through without a problem in this regard.
\end{rem}

Of course our aim is to see that $A \wedge_q B$ is a model for $A \otimes B$. For this, we want to see that the functor $A \wedge_q -$ is left Quillen if $A$ is cofibrant; and similarly for $- \wedge_q B$. Since colimits of spectra are formed levelwise and the smash product with a simplicial set commutes with colimits, $A \wedge_q -$ commutes with colimits; this already implies that $A \wedge_q -$ is a left adjoint. Also, since the simplicial smash product on $\Sp$ is defined levelwise, it also commutes with the simplicial smash product. Hence we can consider the associated $S^1$-cospectrum $X$ and check whether it is a stable frame.  First, we have to see that $X_n = A \wedge_q F_n S^0$ is a cofibrant spectrum, i.e., that all structure maps are cofibrations. In low degrees $A \wedge_q F_n S^0$ is just a point and we have nothing to check; and in high degrees the structure maps are either isomorphisms if one uses the structure maps of $F_n S^0$ which are isomorphisms, or a suspension of the structure maps of $A$ which are cofibrations by our assumption that $A$ is cofibrant; hence $A \wedge_q F_n S^0$ is cofibrant. \\
The structure maps of $X$ are the images of the canonical map $\tau_n: F_n S^1 \rightarrow F_{n-1} S^0$. If one chooses $m$ such that $p(m) > n$, then $(A \wedge_q F_{n+1} S^1)_m$ and $(A \wedge_q F_n S^0)_m$ are isomorphic and the isomorphism is actually induced by $A \wedge _q \tau_n$; hence $A \wedge_q -$ sends $\tau_n$ to a $\pi_*$-isomorphism. (This is the point where the assumption that $p$ is unbounded enters.), and $X$ is a stable frame; so $A \wedge_q - $ is left Quillen. Of course, the same proof applies to $- \wedge B$. \\
We would like to see that $A \wedge_q \U$ is weakly equivalent to $A$. This seems to be difficult to see directly since one cannot write down a map between those two which could be the desired $\pi_*$-isomorphism, and when replacing one of them fibrantly, one loses control; note that it is tempting to use the structure maps of $A$ to write down a map $A \wedge_q \U \rightarrow A$; after all, $(A \wedge_q \U)_n$  is just $A_{q(n)} \wedge S^{p(n)}$, and the structure maps yield a map $A_{q(n)} \wedge S^{p(n)} \rightarrow A_n$. But this is \emph{not a map of spectra} - the various sphere coordinates are used in the wrong order. However, we can use a backdoor approach: The functor $- \wedge_q \U$ is also left Quillen (here we use that $q$ is unbounded), and $\U \wedge_q \U$ is in fact isomorphic to the sphere spectrum - it is in level $n$ isomorphic to $S^n$ and all structure maps are isomorphisms since this is true for $\U$; hence it is a model for $F_0 S^0 = \U$. Note that the levelwise identity does not induce a map $\U \rightarrow \U \wedge_q \U$ since various twists are involved in the structure maps of the latter.
Still, by \ref{MAIN} this means that $- \wedge_q \U$ is naturally weakly equivalent to the identity (maybe through a zig-zag),  and hence $A \wedge_q \U$ is weakly equivalent to $A$. Altogether, we obtain the following:

\begin{thm}
The functor $\Sp \times \Sp \rightarrow \Sp$, $(A,B) \rightarrow A \wedge_q B$, represents the smash product functor $\SHC \times \SHC \rightarrow \SHC$. 
\end{thm}

\begin{proof}
This is clear by the preceding discussion and the construction of the smash product via Quillen functors. Note that a map $f: A \rightarrow A'$ induces a natural transformation $A \wedge_q - \rightarrow A' \wedge_q -$.
\end{proof}

This can be used to give a direct proof that the smash product is symmetric: After all, $A \wedge_q B \cong B \wedge_p A$, and both are models for the smash product $A \wedge B$. The following, similar statement is also interesting in its own right:

\begin{prop}
Given two left Quillen functors $F, G: \Sp \rightarrow \Sp$, the derived functors satisfy $FG \cong GF$.
\end{prop}

\begin{proof}
The derived functors of $F$ and $G$ are determined up to isomorphism by $A = F(\U)$ and $B = G(\U)$; hence we may assume $F = A \wedge_q -$ and $G = B \wedge_p -$. To see that the derived functors $FG$ and $GF$ are isomorphic, it suffices to see that $F(G(\U)) \cong G(F(\U))$; however, $F(G(\U)) \cong F(B) = A \wedge_q B \cong B \wedge_p A \cong G(A) \cong G(F(\U))$ as desired.
\end{proof}

\subsection{The original definition of the smash product}

We will follow \cite{adams} in our description of the smash product. For this, it is necessary to work with spectra over a suitable closed monoidal category of pointed topological spaces; all the results we had about simplicial model categories carry over to the topological setting, in particular, there is a notion of topological stable frame in a topological model category. We denote the topological category of spectra as $\Sp^{\Top}$. \\
The basic idea is very similar to the one outlined above, with one subtle difference regarding the structure maps. Choose functions $p,q: \mathbb{N} \rightarrow \mathbb{N}$ as above. Given two topological spectra $A$ and $B$, we again define a spectrum $A \wedge_q B$ with $n$-th space $A_{p(n)} \wedge B_{q(n)}$, but with slightly different structure maps. In the topological setting, there is a "multiplication by -1"-map $\tau$ on $S^1$; regarding $S^1$ as the one-point compactification of $\mathbb{R}$, this is just the map sending $x$ to $-x$. Now, If $q(n+1) = q(n)+1$, we just use the structure map of $B$ to obtain a map  $A_{p(n)} \wedge B_{q(n)} \wedge S^1 \rightarrow A_{p(n+1)} \wedge B_{q(n+1)} $; however, if $q(n+1) = q(n)$, we first permute the $S^1$ past the $B_{q(n)}$, then use  $\tau$ to obtain a self-map $A_{p(n)} \wedge S^1 \wedge B_{q(n)} \rightarrow  A_{p(n)} \wedge S^1 \wedge B_{q(n)}$, and then use the structure map of $A$ to obtain a map to  $A_{p(n+1)} \wedge B_{q(n+1)}$.  The \emph{classical smash product} $- \wedge -: \Ho(\Sp^{\Top}) \times \Ho(\Sp^{\Top}) \rightarrow \Ho(\Sp^{\Top})$ constructed in \cite{adams} has the following basic property:

\begin{thm}
For arbitrary $p$, $q$, there is a natural isomorphism $A \wedge_q B \rightarrow A \wedge B$.
\end{thm}

Using topological stable frames, we obtain the following:

\begin{thm}
For cofibrant spectra $A$, $B$, the functors $A \wedge_q -: \Sp^{\Top} \rightarrow \Sp^{\Top}$ and $- \wedge_q B: \Sp^{\Top} \rightarrow \Sp^{\Top}$ are left Quillen and send the topological sphere spectrum $\U^{\Top}$ to $A$ resp. $B$ up to weak equivalence.
\end{thm}

\begin{proof}
Basically, the same proof applies. The functor $A \wedge_q -$ commutes with colimits and with the smash product of topological spaces, hence is a left adjoint. The spectrum $A \wedge F_n S^0$ is cofibrant, and the maps $F_n S^1 \rightarrow F_{n-1} S^1$ induce $\pi_*$-isomorphisms since they induce isomorphisms in high enough degrees.  That $A \wedge_q \U^{\Top} \cong A$ in the homotopy category follows as above since $\U^{\Top} \wedge_q \U^{\Top}$ is isomorphic to $\U^{\Top}$; the identification of this with the sphere spectrum will be slightly different from the above identification because of the signs involved in the structure maps, but it is still isomorphic to the sphere spectrum.
\end{proof}

The associativity and commutativity isomorphisms for the smash product are then obtained by making intelligent choices for $p$ and $q$. In particular, things are arranged such that the associativity, unit and commutativity isomorphisms stem from natural transformations of the overlying Quillen functors. Thus we  may apply \ref{monoid} (or, rather, the remark following the proof) to obtain the following:

\begin{thm}
Let $F: \Sp \rightarrow \Sp^{\Top}$ denote the geometric realization functor. Then $F$ induces a monoidal equivalence $\Ho(\Sp) \rightarrow \Ho(\Sp^{\Top})$ where the first category is equipped with the smash product we have constructed and $\Ho(\Sp^{\Top})$ is equipped with the classical smash product just described.
\end{thm}

\bibliographystyle{alpha}
\bibliography{bib}
\end{document}